\documentclass[12pt]{amsart}

\usepackage{amsthm,amsfonts,amsmath,amssymb,latexsym,epsfig,mathrsfs,yfonts,marvosym}
\usepackage[usenames]{color}
\usepackage{graphicx}
\usepackage[all]{xy}
\usepackage{epsfig}
\usepackage{epic}
\usepackage{eepic}
\usepackage{hyperref}
\usepackage{verbatim}
\usepackage{dsfont}\let\mathbb\mathds

\DeclareMathAlphabet\oldmathcal{OMS}        {cmsy}{b}{n}
\SetMathAlphabet    \oldmathcal{normal}{OMS}{cmsy}{m}{n}
\DeclareMathAlphabet\oldmathbcal{OMS}       {cmsy}{b}{n}
 
\usepackage{eucal}

\usepackage[active]{srcltx}

\newtheorem{theorem}{Theorem}[section]
\newtheorem{lemma}[theorem]{Lemma}
\newtheorem{proposition}[theorem]{Proposition}
\newtheorem{corollary}[theorem]{Corollary}
\newtheorem{definition}[theorem]{Definition}

\newenvironment{example}{\medskip \refstepcounter{theorem}
\noindent  {\bf Example \thetheorem}.\rm}{\,}
\newenvironment{remark}{\medskip \refstepcounter{theorem}
\newcommand{\mute}[2] {}
\newcommand     {\printname}[1] {}

\noindent  {\bf Remark \thetheorem}.\rm}{\,}
\newtheorem{thm}[theorem]{Theorem}
\newtheorem{defn}[theorem]{Definition}

\newtheorem{prop}[theorem]{Proposition}

\newtheorem{Thm}{Theorem}
\newtheorem{Cor}[Thm]{Corollary}

\newcommand{\bb}{\mathbb}
\newcommand{\delbar}{\bar{\partial}}

\def\d{\partial}

\def\<{\langle}

\def\>{\rangle}                                    
\def\a{\alpha}                                     

\def\g{\gamma}                                     

\def\BOne{{\mathchoice {\rm 1\mskip-4mu l} {\rm 1\mskip-4mu l}
                          {\rm 1\mskip-4.5mu l} {\rm 1\mskip-5mu l}}}

\def\fract#1#2{\raise4pt\hbox{$ #1 \atop #2 $}}
\def\decdnar#1{\phantom{\hbox{$\scriptstyle{#1}$}}
\left\downarrow\vbox{\vskip15pt\hbox{$\scriptstyle{#1}$}}\right.}
\def\decupar#1{\phantom{\hbox{$\scriptstyle{#1}$}}
\left\uparrow\vbox{\vskip15pt\hbox{$\scriptstyle{#1}$}}\right.}

\def\bbc{{\mathbb C}}

\def\bbp{{\mathbb P}}
\def\bbq{{\mathbb Q}}
\def\bbr{{\mathbb R}}

\def\bbz{{\mathbb Z}}

\def\gra{\alpha}
\def\grb{\beta}

\def\grd{\delta}

\def\grg{\gamma}
\def\gri{\iota}
\def\grk{\kappa}

\def\gro{\omega}

\def\grr{\rho}

\def\grt{\tau}

\def\grz{\zeta}
\def\grD{\Delta}

\def\grO{\Omega}

\def\grS{\Sigma}

\def\bfa{{\bf a}}

\def\bfp{{\bf p}}

\def\bfw{{\bf w}}
\def\bfx{{\bf x}}

\def\bfz{{\bf z}}

\def\cala{{\mathcal A}}
\def\calb{{\mathcal B}}
\def\calc{{\mathcal C}}
\def\calo{{\mathcal O}}

\def\cald{{\mathcal D}}

\def\calf{{\mathcal F}}

\def\calj{{\mathcal J}}

\def\calo{{\mathcal O}}

\def\calr{{\mathcal R}}
\def\cals{{\oldmathcal S}}

\def\calw{{\mathcal W}}

\def\calx{{\mathcal X}}

\def\la#1{\hbox to #1pc{\leftarrowfill}}
\def\ra#1{\hbox to #1pc{\rightarrowfill}}
\def\calz{{\oldmathcal Z}}

\def\ga{{\mathfrak a}}

\def\gc{{\mathfrak c}}

\def\gf{{\mathfrak f}}
\def\gg{{\mathfrak g}}

\def\gi{{\mathfrak i}}

\def\gm{{\mathfrak m}}
\def\gn{{\mathfrak n}}
\def\go{{\mathfrak o}}

\def\gr{{\mathfrak r}}

\def\gt{{\mathfrak t}}

\def\gB{{\mathfrak B}}
\def\gC{{\mathfrak C}}
\def\gD{{\mathfrak D}}

\def\gG{{\mathfrak G}}
\def\gH{{\mathfrak H}}

\def\gR{{\mathfrak R}}

\def\hook{\mathbin{\hbox to 6pt{%
                 \vrule height0.4pt width5pt depth0pt
                 \kern-.4pt
                 \vrule height6pt width0.4pt depth0pt\hss}}}
\def\d{\partial}
\def\bA{\bar{A}}

\def\bp{\bar{p}}

\def\tu{\tilde{u}}

\def\teta{\tilde{\eta}}

\def\tvarpi{\tilde{\varpi}}

\newcommand{\ol}{\overline}
\newcommand{\ul}{\underline}

\begin{document}

\title{On the Equivalence Problem for Toric Contact Structures on ${\bf S^3}$-bundles over ${\bf S^2}$ }

\author{Charles P. Boyer}
\address{Department of Mathematics and Statistics,
University of New Mexico, Albuquerque, NM 87131.}
\author{Justin Pati}
\address{Matematiska Institutionen, Uppsala Universitet, Box 480,
751 06 Uppsala, Sweden}

\email{cboyer@math.unm.edu, justin@math.uu.se} 
\thanks{The second author was partially supported by the G{\"o}ran Gustafsson Foundation for Research in Natural Sciences and Medicine.}


\maketitle

\markboth{Equivalence Problem for Contact Structures on ${\bf S^3}$-bundles over ${\bf S^2}$ }{Charles Boyer and Justin Pati}

\begin{abstract}
We study the contact equivalence problem for toric contact structures on $S^3$-bundles over $S^2$. That is, given two toric contact structures, one can ask the question: when are they equivalent as contact structures while inequivalent as toric contact structures? In general this appears to be a difficult problem. To find inequivalent toric contact structures that are contact equivalent, we show that the corresponding $3$-tori belong to distinct conjugacy classes in the contactomorphism group. To show that two toric contact structures with the same first Chern class are contact inequivalent, we use Morse-Bott contact homology. We treat a subclass of contact structures which include the Sasaki-Einstein contact structures $Y^{p,q}$ studied by physicists in \cite{GMSW04a,MaSp05b,MaSp06}. In this subcase we give a complete solution to the contact equivalence problem by showing that $Y^{p,q}$ and $Y^{p'q'}$ are inequivalent as contact structures if and only if $p\neq p'$.
\end{abstract}

\tableofcontents

\section*{Acknowledgements}
The first author thanks Miguel Abreu, Vestislav Apostolov, David Calderbank, Paul Gauduchon, and Christina T{\o}nnesen-Friedman for helpful discussions on toric geometry.
The second author would like to thank Tobias Ekholm, Georgios Dimitroglou Rizell, and Clement Hyrvier for many useful discussions. 
\section*{Introduction}



It is well known that contact structures have only discrete invariants, that is, Gray's Theorem says that the deformation theory is trivial.  Apparently, the crudest such invariant is the first Chern class of the contact bundle $\cald$. Indeed, the mod $2$ reduction of $\cald$ is a topological invariant, namely the second Stiefel-Whitney class. A much more subtle and powerful invariant is contact homology, a small part of the more general 
symplectic field theory (SFT) of Eliashberg, Givental, and Hofer \cite{ElGiHo00} 
- which can be used to distinguish contact structures belonging to the same isomorphism class of oriented $2n$-plane bundle. 

On the other hand given two contact structures with the same invariants, when can one show that they are equivalent. In full generality this appears to be a very difficult problem. However, if we restrict ourselves to toric contact structures in dimension five, we can begin to get a handle on things. The problem is of particular interest when applied to toric contact manifolds since they have been classified \cite{Ler02a}. Thus, one is interested in when two inequivalent toric contact structures are equivalent as contact structures. Specializing further we consider all toric contact structures on $S^3$-bundles over $S^2$. It is well known that such manifolds are classified by $\pi_1(SO(4))=\bbz_2$, so there are exactly two such bundles, the trivial bundle $S^2\times S^3$ and one non-trivial bundle  $X_\infty$ (in Barden's notation \cite{Bar65}). They are distinguished by their second Stiefel-Whitney class $w_2\in H^2(M,\bbz_2)$. The problem of determining when two such toric contact structures belong to equivalent contact structures is now somewhat tractible owing to the work of Karshon \cite{Kar03} and Lerman \cite{Ler03b}. In turn their work is based on the important observation of Gromov \cite{Gro85} of how the topology of the symplectomorphism group of the symplectic structures $\gro_{k,l}$ on $S^2\times S^2$ changes as the `distance' between $k$ and $l$ changes. 

The general toric contact structures on $S^2\times S^3$ or $X_\infty$ depend on four integers $(p_1,p_2,p_3,p_4)$ which satisfy $\gcd(p_i,p_j)=1$ for $i=1,2$ and $j=3,4$. We write the contact structures as $\cald_\bfp$ using vector notation $\bfp$ for the quadruple. However, this general situation appears somewhat intractable, so we consider the special case when one of the two pairs of integers $(p_1,p_2)$ or $(p_3,p_4)$ are equal. What makes this sub-case more tractable is that a certain quotient is a Hirzebruch surface with branch divisors. To treat the sub-case we can assume that $p_3=p_4$. It is often convenient to further divide this case in two sub-cases, namely, we set $\bfp=(j,2k-j,l,l)$ for $S^2\times S^3$ and $\bfp=(j,2k-j+1,l,l)$ for $X_\infty$ with $1\leq j\leq k$. We denote either one of these contact structures by $\cald_{p_1,p_2,l,l},\cald_{j,2k-j,l,l},\cald_{j,2k-j+1,l,l}$ or simply as $\cald_{j,k,l}$ depending on which notation is more convenient. Since the first Chern class $c_1(\cald)$ of the contact bundle clearly distinguishes contact structures, we see that from contact homology in our case $p_1+p_2$ or equivalently $k$ is a contact invariant when $c_1(\cald)$ is fixed.  Our main result about inequivalence is

\begin{Thm}\label{mainthm1}
Two toric contact structures $\cald_{p_1,p_2,l,l}$ and $\cald_{p'_1,p'_2,l',l'}$ on $S^2\times S^3$ or $X_\infty$ are inequivalent contact structures if $p_1'+ p_2 '\neq (p_1 + p_2)$.
\end{Thm}

For our main result about equivalence, we need to specialize a bit further. In this case we require that $\gcd(p_2-p_1,l)$ be constant. We have
\begin{Thm}\label{mainthm2}
The two contact structures  $\cald_{p_1,p_2,l,l}$ and $\cald_{p'_1,p'_2,l,l}$ satisfying $p'_1+p'_2=p_1+p_2$ are equivalent if $\gcd(l,p_2-p_1)=\gcd(l,p'_2-p'_1)$. 
\end{Thm}

Recently there has been a great deal of focus on certain toric contact structures $Y^{p,q}$ with vanishing first Chern class on $S^2\times S^3$ discovered by Gauntlett, Martelli, Sparks, and Waldram \cite{GMSW04a}, and used in their study of the AdS/CFT conjecture \cite{GaMaSpWa04c,GaMaSpWa05} (see Chapter 11 of \cite{BG05} and \cite{Spa10} and references therein). Our results give a complete answer to the contact equivalence problem for these structures.

\begin{Thm}\label{SEequiv}
Let $\phi$ denote the Euler phi function. The toric contact structures $Y^{p,q}$ and $Y^{p',q'}$ on $S^2\times S^3$ belong to equivalent contact structures if and only if $p'=p$, and for each fixed integer $p>1$ there are exactly $\phi(p)$ toric contact structures $Y^{p,q}$ on $S^2\times S^3$ that are equivalent as contact structures, and for each such contact structure $\cald_p$ there are $\phi(p)$ compatible Sasaki-Einstein metrics that are inequivalent as Riemannian metrics.  Moreover, the contactomorphism group of $\cald_p$ has at least $\phi(p)$ conjugacy classes of maximal tori of dimension three.
\end{Thm}

A partial result, namely that, $Y^{p',1}$ and $Y^{p,1}$ are inequivalent contact structures if $p'\neq p$, was recently given by Abreu and Macarini \cite{AbMa10}, and an outline of the proof of Theorem \ref{SEequiv} was recently given by one of us \cite{Boy11}.

As a bonus we also obtain the following results concerning extremal Sasakian structures:

\begin{Cor}\label{extcor}
For both $S^2\times S^3$ and $X_\infty$ the moduli space of extremal Sasakian structures has a countably infinite number of components. Moreover, each component has extremal Sasakian metrics of positive Ricci curvature whose isometry group contains $T^3$.
\end{Cor}

This corollary follows already from the results of \cite{Pat09,Pat10} and \cite{Boy10b}, but Theorem \ref{mainthm1} actually gives a much larger class. As shown in \cite{Boy10b} many of these components are themselves non-Hausdorff.

\begin{Cor}\label{SEcor}
The moduli space of Sasaki-Einstein metrics on $S^2\times S^3$ has a countably infinite number of components. Moreover, each such component has Sasaki-Einstein metrics whose isometry group contains $T^3$.
\end{Cor}

This corollary also follows from \cite{AbMa10}, but Theorem \ref{mainthm1} gives a larger class.

\section{Contact Structures and Cones} 

It is well known that contact geometry is equivalent to the geometry of certain symplectic cones. However, for certain contact structures there are several cones that become important, and as we shall see they are all related. 

\noindent {\it A warning about notation}: In contact topology the contact bundle is usually denoted by $\xi$, whereas, in Sasakian geometry $\xi$ is almost always a Reeb vector field. To avoid confusion we eschew the use of $\xi$ completely, and use $\cald$ for the contact bundle and $R$ for a Reeb vector field.

\subsection{Contact Structures}
Recall that a {\it contact structure}\footnote{This is not the most general definition of a contact structure, but it suffices in most situations (cf. \cite{BG05}), and certainly for us here.} on a connected oriented manifold $M$ is an equivalence class of 1-forms $\eta$ satisfying $\eta\wedge (d\eta)^n\neq 0$ everywhere on $M$ where two 1-forms $\eta,\eta'$ are equivalent if there exists a nowhere vanishing function $f$ such that $\eta'=f\eta$. We shall also assume that our contact structure has an orientation, or equivalently, the function $f$ is everywhere positive. More conveniently the contact structure can be thought of as the oriented $2n$-plane bundle defined by $\cald=\ker\eta$, and we denote by $\gC^+(\cald)$ the set of all contact 1-forms representing the oriented bundle $\cald$. 

We denote by $\cala\calc(\cald)$ the set of almost complex structures $J$ on $\cald$ that are compatible with the contact structure in the sense that the following two conditions hold for any smooth sections $X,Y$ of $\cald$ 
$$d\eta(JX,JY)=d\eta(X,Y), \qquad d\eta(JX,Y)>0.$$
It is easy to see that these conditions are independent of the choice of 1-form $\eta$ representing $\cald$. Notice that the pair $(\cald,J)$ defines a strictly pseudoconvex almost CR structure on $M$, and a choice of contact form $\eta$ gives a choice of Levi form essentially $d\eta$. 

Also for every choice of contact 1-form $\eta$ there exists a unique vector field $R$, called the {\it Reeb vector field}, that satisfies $\eta(R)=1$ and $R\hook d\eta=0$. Such vector fields and the orbits of their flows will play a crucial role for us. We can now extend $J$ to an endomorphism $\Phi$ of $TM$ by defining $\Phi|_\cald =J$ and $\Phi R=0$. The triple $(R,\eta,\Phi)$ canonically defines a Riemannian metric on $M$ by setting $g=d\eta\circ (\Phi\otimes \BOne) +\eta\otimes\eta$, and the quadruple $(R,\eta,\Phi,g)$ is known as a {\it contact metric structure} on $M$.

Notice that $R$ defines a one dimensional foliation $\calf_R$ on $M$, often called the {\it characteristic foliation}. We say that the foliation $\calf_R$ is {\it quasi-regular} if there is a positive integer $k$ such that each point has a
foliated coordinate chart $(U,x)$ such that each leaf of
$\calf_R$ passes through $U$ at most $k$ times. If $k=1$ then
the foliation is called {\it regular}. We also say that the corresponding contact 1-form $\eta$ is {\it quasi-regular (regular)}, and more generally that a contact structure $\cald$ is {\it quasi-regular (regular)} if it has a quasi-regular (regular) contact 1-form. A contact 1-form (or characteristic foliation) that is not quasi-regular is called {\it irregular}. On a compact manifold any quasi-regular contact form is necessarily K-contact, and then the foliation $\calf_R$ is equivalent to a locally free circle action \cite{BG05} which preserves the quadruple $(R,\eta,\Phi,g)$. It is this case that we are interested in, and the quotient space $\calz=M/\calf_R$ is a compact orbifold with a naturally defined symplectic structure $\gro$ and compatible almost complex structure $\hat{J}$ satisfying $\pi^*\gro=d\eta$ and $J$ is the horizontal lift of $\hat{J}$, that is, $(\gro,\hat{J})$ defines an almost K\"ahler structure on the orbifold $\calz$. Moreover, $\eta$ can be interpreted as a connection 1-form in the principal $S^1$ orbibundle $\pi:M\ra{1.6} \calz$ with curvature 2-form $\pi^*\gro$. 

This construction has a converse, that is, beginning with a compact almost K\"ahler orbifold one can construct a K-contact structure on the total space of certain $S^1$ orbibundle over $\calz$. This is often referred to as the {\it orbifold Boothby-Wang construction}. In \cite{ElGiHo00} a contact manifold constructed in this way is called a {\it prequantization space}. In this paper we are interested in the case when both $J$ and $\hat{J}$ are integrable. Then the quadruple $(R,\eta,\Phi,g)$ is a {\it Sasakian structure} on $M$, and $(\gro,\hat{J})$ defines a is projective algebraic orbifold structure on $\calz$ with an orbifold K\"ahler metric.

\subsection{Orbifolds}
As described in the preceding paragraph, orbifolds will play an important role for us in this paper. We refer to Chapter 4 of \cite{BG05} for the basic definitions and results. Here we want to emphasize several aspects. First, many cohomology classes that are integral classes on manifolds are only rational classes on the underlying topological space of an orbifold, in particular, the orbifold first Chern class of a complex line orbibundle or circle orbibundle is generally a rational class. However, not all rational classes occur as such. In order to determine which rational classes can be used to classify line orbibundles, it is convenient to pass to Haefliger's classifying space $B\calx$ \cite{Hae84,BG05} of an orbifold $\calx$ where, as with smooth manifolds, all complex line orbibundles correspond to integral cohomology classes. Let $\calx$ be a complex orbifold with underlying topological space $X$. Then Haefliger's orbifold cohomology $H^*_{orb}(\calx,\bbz)$ equals $H^*(B\calx,\bbz)$ which is generally different than $H^*(X,\bbz)$, but satisfies $H^*_{orb}(\calx,\bbz)\otimes \bbq=H^*(X,\bbz)\otimes \bbq$. So, for example, we an obtain integral cohomology class $p^*c_1^{orb}(\calx)\in H^2_{orb}(\calx,\bbz)$ for complex line orbibundles from the rational class $c_1^{orb}(\calx)\in H^2(X,\bbq)$. This amounts to clearing the order of the orbifold in the denominator. Here $p:B\calx\ra{1.6} \calx$ is the natural projection. We warn the reader that the orbifold cohomology $H^*_{orb}(\calx,\bbz)$ is not Chen-Ruan cohomology.

The orbifolds that occur in this paper are of a special type. They are all complex orbifolds whose underlying space is a smooth projective algebraic variety with an added orbifold structure. In such cases it is convenient to view an orbifold $\calx$ as a pair $(X,\grD)$ where $X$ is a smooth algebraic variety and $\grD$ is a certain $\bbq$-divisor, called a {\it branch divisor} \cite{BGK05,BG05,GhKo05}. We write $(X,\emptyset)$ to denote the algebraic variety $X$ with the trivial orbifold structure, that is the charts are just the standard manifold charts. In this situation as emphasized in \cite{GhKo05} we consider the map $\BOne_X:(X,\grD)\ra{1.6} (X,\emptyset)$ which is the identity as a set map, and a Galois cover with trivial Galois group.

\subsection{Symplectic Cones}
Given a contact structure $\cald$ on $M$ we recall the symplectic cone $C(M)=M\times \bbr^+$ with its natural symplectic structure (the {\it symplectization} of $(M,\cald)$) $\grO=d(r^2\eta)$ where $r$ is a coordinate on $\bbr^+$. Suppose $\eta'=e^{2f}\eta$ is another contact 1-form in $\gC^+(\cald)$, then changing coordinates $r'=e^{-f}r$ we see that 
$$d(r'^2\eta')=d(r^2e^{-2f}e^{2f}\eta)=d(r^2\eta)=\grO,$$
so the symplectic structure $\grO$ on $C(M)$ depends only on the contact structure $\cald$.
Recall the Liouville vector field $\Psi=r\frac{\d}{\d r}$ on the cone $C(M)$ and notice that $$\Psi=r\frac{\d}{\d r}=r'\frac{\d}{\d r'}.$$
We have chosen the dependence of $\grO$ on the radial coordinate to be homogeneous of degree $2$ with respect to $\Psi$, since we want compatibility with cone metrics and these are homogeneous of degree $2$. 

Now for each choice of contact form $\eta\in\gC^+(\cald)$ there is a natural extension of the almost complex structure $J$ on $\cald$ to an almost complex structure $I$ on the cone $C(M)$ defined uniquely by
\begin{equation}\label{coneMstructure}
I=\Phi +\Psi\otimes\eta, \qquad I\Psi=-R
\end{equation}
where $\Phi$ is the extension of $J$ to $TM$ defined by $\Phi R=0$. The following is well known \cite{BG05} and straightforward to verify

\begin{lemma}\label{Icompat}
Let $(\cald,J)$ be a strictly pseudoconvex almost CR structure such that $\cald$ is a contact structure on $M$, and let $(C(M),\grO))$ be its symplectic cone. Then for each contact form $\eta\in\gC^+(\cald)$, the corresponding almost complex structure on $C(M)$ satisfies
\begin{enumerate}
 \item $\pounds_\Psi I=0$.
\item $I$ is compatible with $\grO$, that is, for any vector fields $X,Y$ on $C(M)$ we have $\grO(IX,IY)=\grO(X,Y)$ and $\grO(IX,Y)>0$ defines a cone metric $\bar{g}=dr^2+r^2g$.
\item The cone metric $\bar{g}$ on $C(M)$ is almost K\"ahler.
\item The contact form $\eta$ is K-contact if and only if the vector field $\Psi-iR$ is pseudoholomorphic with respect to $I$, or equivalently $\pounds_R I=\pounds_R \Phi=0$.
\item If $(\cald,J)$ is an integrable CR structure, then $I$ is integrable if and only if $\pounds_R I=\pounds_R \Phi=0.$ In this case, $(R,\eta,\Phi,g)$ defines a Sasakian structure on $M$ with metric $g=d\eta\circ(\Phi\otimes \BOne)+\eta\otimes \eta$, and $(\grO,I)$ defines a complex structure on $C(M)$ whose cone metric $dr^2+r^2g$ is K\"ahler.
\end{enumerate}

\end{lemma}

There is a converse to this result, namely: Let $(M\times \bbr^+,\grO)$ be a symplectic cone such that $\grO$ is homogeneous of degree two, then $d(\Psi\hook \grO)=\pounds_\Psi\grO=2\grO$, so $\grO$ is exact, and we can define 
$$\teta=\frac{1}{r^2}\Psi\hook \grO.$$
It is then easy to check that $\teta$ is the pullback of a contact 1-form $\eta$ on $M$, and by changing coordinates $r'=e^{-f}r$ we can accommodate all the contact forms in $\gC^+(\cald)$. We can also check that there is a 1-1 correspondence between the compatible almost complex structures $I$ on $C(M)$ and elements of $\gC^+(\cald)$, and that Equations (\ref{coneMstructure}) hold, so we recover the full contact metric structure for each $\eta\in \gC^+(\cald)$. Summarizing we have the correspondences: \bigskip
\newline (1) symplectic cone $(C(M),\grO)$ $\leftrightarrow$ contact structure $(M,\cald)$; 
\newline (2) almost K\"ahler cone $(C(M),\grO,I)$ $\leftrightarrow$ contact metric structure \break $(M,R,\eta,\Phi,g)$;
\newline (3) almost K\"ahler cone $(C(M),\grO,I)$ with $\Psi-iR$ pseudoholomorphic $\leftrightarrow$ K-contact structure $(M,R,\eta,\Phi,g)$;
\newline (4) K\"ahler cone $(C(M),\grO,I)$ with $\Psi-iR$ holomorphic $\leftrightarrow$ Sasakian structure $(M,R,\eta,\Phi,g)$.

\subsection{Sasakian Structures} 
The contact structures considered in this paper are all of {\it Sasaki type}, that is, there is a contact form $\eta$ and compatible metric $g$ such that $\cals=(R,\eta,\Phi,g)$ is a Sasakian structure on $M$. In this case not only is the cone $C(M)$ discussed above K\"ahler, but the geometry transverse to the characteristic foliation $\calf_R$ is also K\"ahler. This gives rise to a basic cohomology ring $H^*_B(\calf_R)$ (see Section 7.2 of \cite{BG05}), and a transverse Hodge theory. This gives basic Chern classes $c_k(\calf_R)$ which if $(R,\eta,\Phi,g)$ is quasi-regular, are the pullbacks of the orbifold Chern classes $c_k^{orb}(\calz)$ on the base orbifold $\calz$. In particular we are interested in the basic first Chern class $c_1(\calf_R)\in H^{1,1}_B(\calf_R)$. A Sasakian structure $\cals=(\xi,\eta,\Phi,g)$ is said to be {\it positive (negative)} if its basic first Chern class $c_1(\calf_\xi)$ can be represented by a positive (negative) definite $(1,1)$-form. It is {\it null} if $c_1(\calf_\xi)=0$, and {\it indefinite} otherwise. It follows from Lemma 5.1 of \cite{Boy10b} that all Sasakian structures occurring in toric contact structures of Reeb type are either positive or indefinite. 

Recall the {\it transverse homothety} (cf. \cite{BG05}) taking a Sasakian structure $\cals=(\xi,\eta,\Phi,g)$ to the Sasakian structure $$\cals_a=(a^{-1}\xi,a\eta,\Phi,ag+(a^2-a)\eta\otimes \eta)$$
for any $a\in \bbr^+$. Then Theorem 7.5.31 of \cite{BG05} says that if $\cals=(\xi,\eta,\Phi,g)$ is a positive Sasakian structure, that is, $c_1(\calf_\xi)$ can be represented by a positive definite $(1,1)$-form, there exists $a_0\in\bbr^+$ and transverse homotheties $\cals\mapsto \cals_a$ such that the Sasakian metric $g_a$ has positive Ricci curvature for all $a<a_0$.

\subsection{The Sasaki Cone}
Let $\gC\gR(\cald,J)$ denote the group of almost CR transformations of $(\cald,J)$ on $M$. If $M$ is compact, it is a Lie group which is compact except when $(\cald,J)$ is the standard CR structure on the sphere $S^{2n+1}$ by, in various stages, a theorem of Frances, Lee, and Schoen (cf. \cite{Boy10a}). We let $\gc\gr(\cald,J)$ denote the Lie algebra of $\gC\gR(\cald,J)$. Recall \cite{BGS06} that the subset 
$$\gc\gr^+(\cald,J)=\{X\in \gc\gr(\cald,J)~|~\eta(X)>0\}$$
is independent of the choice of $\eta\in\gC^+(\cald)$ and is an open convex cone (without the cone point) in $\gc\gr(\cald,J)$. Now the adjoint action of the group $\gC\gR(\cald,J)$ on its Lie algebra leaves $\gc\gr^+(\cald,J)$ invariant, and the quotient space $$\grk(\cald,J)=\gc\gr^+(\cald,J)/\gC\gR(\cald,J)$$ 
is known as the {\it (reduced) Sasaki cone} of $(\cald,J)$. One should think of $\grk(\cald,J)$ as the {\it moduli space} of K-contact structures associated to the strictly pseudoconvex almost CR structure $(\cald,J)$. In the case that the almost CR structure is integrable, $\grk(\cald,J)$ is the moduli space of Sasakian structures associated to $(\cald,J)$.

It is often convenient to work with the {\it unreduced Sasaki cone} given by choosing a maximal torus $T$ of $\gC\gR(\cald,J)$. Then the unreduced Sasaki cone is $\gt^+(\cald,J)=\gt\cap \gc\gr^+(\cald,J)$ where $\gt$ is the Lie algebra of $T$. The relation between the two Sasaki cones is 
\begin{equation}\label{Sasakicone2}
\grk(\cald,J)= \gt^+(\cald,J)/\calw(\cald,J),
\end{equation}
where $\calw(\cald,J)$ is the Weyl group of $\gC\gR(\cald,J)$. 
$\gt^+(\cald,J)$ is a subspace of the Reeb cone $\calr^+(\cald)$ \cite{Boy10a} which is the subspace of all vector fields on $M$ that is the Reeb vector field of some contact 1-form representing the oriented contact structure $\cald$. 

Of course, many contact structures do not have a Sasaki cone. In fact, a contact structure has a non-empty Sasaki cone if and only if it is of K-contact type. It is important to realize that the Sasaki cone depends on the choice of transverse almost complex structure $J$. Indeed by changing $J$ in a given K-contact structure, we can have more than one Sasaki cone. These occur in bouquets related to the conjugacy classes of maximal tori in the contactomorphism $\gC\go\gn(M,\cald)$ of $(M,\cald)$ \cite{Boy10a,Boy10b}. So a Sasaki bouquet consisting of $N$ Sasaki cones belonging to a contact structure $\cald$ is given by 
\begin{equation}\label{Sasbouq}
\gB_N(\cald)=\cup_{l=1}^N\grk(\cald,J_l).
\end{equation}

\subsection{The Moment Cone} Now let $T$ be a torus subgroup of $\gC\go\gn(M,\cald)$, and let $\gt$ be its Lie algebra. Consider the annihilator $\cald^o$ of $\cald$ which is a trivial real line bundle over $M$. The orientation on $\cald$ allows us to write $\cald^o\setminus \{0\}=\cald_+^o\cup \cald_-^o$, and we can identify $\cald_+^o\approx M\times \bbr^+=C(M)$. Then the contact moment map $\Upsilon:\cald_+^o\ra{1.7} \gt^*$ is defined by 
\begin{equation}\label{mommap}
<\Upsilon(x,p),\grt_x>=<p,\grt_x> 
\end{equation}
where $\grt\in\gt$. 
The {\it moment cone} $C(\Upsilon)$ is defined \cite{Ler02a} as the union of the image set with the cone point, i.e.
\begin{equation}\label{momcone}
C(\Upsilon)=\Upsilon(\cald_+^o)\cup \{0\}.
\end{equation}
This moment map satisfies the following invariance condition, $$\Upsilon(\phi\cdot x,\phi^*p)={\rm Ad}^*_\phi\Upsilon(x,p)=\Upsilon(x,p)$$ 
where $\phi\in T$ and $\phi^*$ denotes the induced action of $T$ on $T^*M$ restricted to $\cald_+^o$. 

By averaging over $T$ we can choose a $T$-invariant contact form $\eta$ which gives an equivariant moment map $\mu_\eta:M\ra{1.7} \gt^*$ satisfying 
\begin{equation}\label{etamommap}
\mu_\eta=\Upsilon\circ \eta.
\end{equation}
Again by averaging we can choose an almost complex structure $J$ that is $T$-invariant, so $\gt$ is an Abelian subalgebra of $\gc\gr(\cald,J)$. Furthermore, the contact form $\eta$ is K-contact (with respect to $J$) if and only if its Reeb vector field $R_\eta$ lies in the Lie algebra $\gt$. In this case we also say that the torus action is of {\it Reeb type} \cite{BG00b}. It is easy to see that this is equivalent to the existence of an element $\grt\in \gt$ such that $\eta(\grt)$ is strictly positive on $M$. When the contact structure $\cald$ is of Reeb type $C(\Upsilon)$ is a convex rational polyhedral cone \cite{Ler02a}.  We have

\begin{lemma}\label{hyperplane}
A $T$-invariant contact form $\eta$ is K-contact if and only if the image $\mu_{\eta}(M)$ lies in the intersection of a hyperplane $H_\eta$ with the moment cone $C(\Upsilon)$. Moreover, in the K-contact case the intersection $P_\eta=H_\eta\cap C(\Upsilon)$ is a simple convex polytope which is rational if and only if $\eta$ is quasi-regular.
\end{lemma}

\begin{proof}
For any contact form $\eta$ with Reeb vector field $R_\eta$ we can extend the moment map (\ref{etamommap}) to a moment map $\mu_\eta :M\ra{1.6} \gc\go\gn(\eta)^*$ associated to the infinite dimensional Lie algebra $\gc\go\gn(\eta)$ of vector fields leaving $\eta$ invariant. The equation
$$\langle\mu_\eta,R_\eta\rangle=\eta(R_\eta)=1$$ 
describes a plane in the infinite dimensional vector space $\gc\go\gn(\eta)^*$.
This plane lies in the subspace $\gt^*$ if and only if $R_\eta$ lies in $\gt$ and this holds if and only if the contact form $\eta$ is K-contact. In this case $R_\eta$ lies in the Sasaki cone $\gt^+(\cald,J)$ which is dual to the interior of the  moment cone. So the plane in $\gt^*$ is a hyperplane.

The positivity of $\gt^+(\cald,J)$ implies that $P_\eta$ is a convex polytope, and the rationality condition follows as in the symplectic case \cite{BG00b}.
\end{proof}

The hyperplane $H_\eta$ is called the {\it characteristic hyperplane}.

\subsection{Equivalence of Contact Structures and Invariants}
We say that two contact manifolds $(M,\cald)$ and $(M',\cald')$ are {\it contactomorphic} or {\it equivalent contact manifolds} if there exists a diffeomorphism $\varphi: M\ra{1.6} M'$ such that $\varphi_*\cald=\cald'$. Alternatively, one can say that for any contact form $\eta'$ representing $\cald'$ the 1-form $\varphi^*\eta'$ is a contact form representing $\cald$. From Gray's Theorem we know that there is no local deformation theory for contact structures.

A contact structure on $M^{2n+1}$ has an underlying {\it almost contact structure} which can be defined as a reduction of the frame bundle to the group $U(n)\times 1$. The homotopy class of almost contact structures is an invariant of the contact structure, and the set $A(M)$ of such homotopy classes is in one-to-one correspondence with the set of homotopy classes of almost complex structures on the cone $C(M)$. The Chern classes of the contact bundle are invariants of the contact structure; in particular, the first Chern class $c_1(\cald)$ will play an important role for us.

Another crude way of distinguishing contact structures is by its type. We say that a contact structure $\cald$ on $M$ is of {\it K-contact type} if there is a 1-form $\eta$ in the contact structure and a choice of almost complex structure $J$ such that $\pounds_R\Phi=0$. We also say that the 1-form $\eta$ is {\it K-contact}. If in addition the almost complex structure $J$ is integrable, the 1-form is a {\it Sasakian} contact form, and the contact structure $\cald$ is said to be of {\it Sasaki type}. We mention that contact structures of K-contact type are symplectically fillable \cite{NiPa07}, while those of Sasaki type are holomorphically fillable \cite{MaYe07}. All contact structures discussed in this paper are of Sasaki type.

Let $(M,\cald)$ and $(M',\cald')$ be compact contact manifolds of K-contact type, then there exist quasi-regular 1-forms $\eta$ and $\eta'$, representing $\cald$ and $\cald'$, respectively such that $(M,\cald)$ and $(M',\cald')$ are orbifold fibrations over symplectic orbifolds $(\calz,\gro)$ and $(\calz',\gro')$, respectively with projection maps $\pi$ and $\pi'$ satisfying $d\eta=\pi^*\gro$ and $d\eta'={\pi'}^*\gro'$. So if $\tilde{\varphi}:(\calz,\gro)\ra{1.8}(\calz',\gro')$ is a symplectomorphism of symplectic orbifolds, we can lift it to an orbibundle map $\varphi:M\ra{1.6} M'$ such that $d\eta=\varphi^*d\eta'$. Since $\varphi$ maps fibers to fibers we can choose it so that $\varphi_*R=R'$. So the symplectomorphism $\tilde{\varphi}$ lifts to a contactomorphism which sends the Reeb vector field of $\eta$ to that of $\eta'$.

Another important invariant of a contact structure is its contactomorphism group $\gC\go\gn(M^{2n+1},\cald)$ defined as the subgroup of the diffeomorphism group $\gD\gi\gf\gf(M)$ that leaves the contact bundle $\cald$ invariant. An invariant within $\gC\go\gn(M^{2n+1},\cald)$ itself is the number $\gn(\cald,n+1)$ of conjugacy classes of $(n+1)$-tori.

Finally, a powerful invariant for distinguishing contact structures is contact homology \cite{ElGiHo00} which we now describe in some detail.

\section{Contact Homology and Pseudoholomorphic Curves}
Here we give an exposition of pseudoholomorphic curve theory as it relates to the Morse-Bott formulation of contact homology.

\subsection{J-holomorphic Curves in Symplectizations}
The study of pseudoholomorphic curves in symplectic manifolds was initiated by Gromov in his seminal paper ~\cite{Gro85}.  Since then, these object have become a basic tool in understanding symplectic geometry and topology. We begin with a definition.  Let $\Sigma$ be a Riemann surface with complex structure $j.$

\begin{definition} 
 A $C^{\infty}$ map, $u$, into the almost complex manifold $(N, J)$ is called $J$-holomorphic if
\begin{equation}
 du + J(u)d(u \circ j) = 0.
\end{equation}
\end{definition}
In this paper we are interested in almost complex structures, $J$,which are \emph{compatible} with the symplectic form.  
Compatibility means that $\omega(v, Jw)$ defines a Riemannian metric on $M.$
In the special case that the symplectic manifold is the symplectization of a contact manifold we are interested
in special almost complex structures and their associated pseudoholomorphic curves.  These are constructed 
in the following way.  Choose a transverse almost complex structure $J_0$ on $\mathcal{D}.$  Then choose
a contact form. Extend $J_0$ to an almost complex structure on the symplectization by declaring  
$$ \label{goodJ} J \frac{\partial}{\partial s} = R_{\eta}$$ 
i.e., by declaring the Reeb vector field $\eta$ to be the imaginary part of the complex line bundle spanned by the $R_{\eta}$ 
and the real direction $\frac{\d}{\d s}$.  Notice, now, that cylinders over periodic Reeb orbits are $J$-holomorphic,
as are their branched covers.  In the rest of the paper, unless otherwise mentioned, we will abbreviate pseudoholomorphic
with the word holomorphic.  In situations where classical holomorphicity comes up it shall be made clear.   

\subsection{Morse-Bott Contact Homology}

Contact homology is a small part of the larger symplectic field theory of Eliashberg, Givental, and Hofer ~\cite{ElGiHo00}.  
The original idea, inspired by Floer homology, was to create a homology theory from the chain complex generated by closed orbits of the 
Reeb vector field.  Given a contact manifold $(M, \mathcal{D})$, we choose a contact form, $\eta$ for $\mathcal{D},$ 
and an almost complex structure $J$ on the symplectization of $M$ which extends the almost complex structure on 
$\mathcal{D}$ such that the Reeb vector field is the purely imaginary direction. Now consider the contact version 
of the symplectic action functional $\mathcal{A}: C^{\infty}(S^1; M) \to \bb{R},$ defined by
\begin{equation} 
\cala(\gamma) = \int_{\gamma} \eta.\label{eq:actfun}
\end{equation}
The critical points of $\mathcal{A}$ are closed orbits of the Reeb vector field of $\eta,$ and the gradient trajectories, when considered 
as living in the symplectization of $M$, are $J$-holomorphic curves from a twice punctured 2-sphere into the symplectization which are 
\emph{asymptotically cylindrical} over closed Reeb orbits, i.e., they are curves for which there exist polar coordinates about each puncture, 
such that for the radius sufficiently small the curve behaves like a cylinder over a closed Reeb orbit.  
Notice also that the value of $\mathcal{A}$, on $\gamma$ is the period of $\gamma$ ~\cite{ElGiHo00} ~\cite{HoWyZe96}. 

We would like to think of this action functional as a Morse function and then construct a Morse-Smale-Witten complex, which should then 
give us information about $M$, $\mathcal{D}$, and the loop space of $M.$  The naive idea does not exactly work, since this action functional
has infinite dimensional kernel and cokernel, so we must search further for a reasonable version of the Morse index.  
For this we use the Conley-Zehnder index, or more generally the Robbin-Salamon index ~\cite{RoSa93}.
The Robbin-Salamon index associates to each path of symplectic matrices a half integer, it is a generalization of 
the Conley-Zehnder index to a more general class of paths of symplectic matrices. This particular definition originally
 appeared in ~\cite{RoSa93}.  This index determines the grading for the chain complex in contact homology.  The 
Robbin-Salamon index should be thought of  as analogous to the Morse index for a Morse function.  The analogy is not perfect, since 
the actual Morse theory we consider should give information about the loop space of the contact manifold.

\begin{remark} The Maslov index can be understood as an  invariant of loops of Lagrangian subspaces in the Grassmanian of Lagrangian 
subspaces of a symplectic vector space V.  In this setting the Maslov index is the intersection number of a path of Lagrangian subspaces 
with a certain algebraic variety called the Maslov cycle.  This is of course related to the Robbin-Salamon and Conley-Zehnder indices of 
a path of symplectic matrices, since we can consider a path of Lagrangian subspaces given by the path of graphs of the desired path of 
symplectic matrices. For more information on this see ~\cite{McDSa}.
\end{remark}

For a symplectic vector bundle $E$ over a Riemann surface $\Sigma$, there is symplectic definition of the first Chern number 
$\langle c_{1}(E),\Sigma \rangle$.  It turns out that this Chern number is the loop Maslov index of a certain loop of symplectic matrices, 
obtained from local trivializations of $\Sigma$ decomposed along a curve $\gamma \subset \Sigma$.  This Chern number agrees with the usual 
definition, considering $E$ as a complex vector bundle which can be obtained via a curvature calculation. Let us consider a symplectic vector space $(V, \omega)$, and let $\Phi(t)$, $t \in[0, T]$ be a path of symplectic 
matrices defined on $V$ starting at the identity such that $det(I - \Phi(T)) \neq 0$.
We call a number $t \in [0,T]$, a \emph{crossing} if $det(\Phi(t) - I)=0.$  
For each crossing we define the \emph{crossing form} 
$$\Gamma(t)v = \omega(v, D \dot{\Phi}(t)v),$$ 
where $\omega$ is the standard symplectic form on $\bb{R}^{2n}.$
A crossing is called \emph{regular} if 
the crossing form is non-degenerate.  One can always homotope a path of symplectic matrices to one with regular 
crossings.

\begin{defn}
 The Conley-Zehnder index of the path $\Phi(t)$ under the above assumptions is given by:
$$\mu_{CZ}(\Phi) = \frac{1}{2}sign(\Gamma(0)) + \sum_{t\neq 0\,, \,t\, a\,\, crossing}sign(\Gamma(t))$$
\end{defn}
The Conley-Zehnder index satisfies the following axioms: 
\begin{enumerate}
\item [i.] (\textbf{Homotopy}) $\mu_{CZ}$ is invariant under homotopies which fix endpoints.
\item [ii.] (\textbf{Naturality}) $\mu_{CZ}$ is invariant under conjugation by paths in $Sp(n, \bb{R}).$
\item [iii.] (\textbf{Loop}) For any path, $\psi$ in $Sp(n, \bb{R}),$ 
and a loop $\phi$, $$\mu_{CZ}(\psi \cdot \phi) = \mu_{CZ}(\psi) + \mu_{l}(\phi),$$  where $\mu_{l}$ is the Maslov index for loops of 
symplectic matrices.
\item [iv.] (\textbf{Direct Sum})  If $n = n' + n''$ and $\psi_{1}$ is a path in $Sp(n', \bb{R})$ 
and $\psi_{2}$ is a path in $Sp(n'', \bb{R})$ 
then for the path $\psi_{1} \oplus \psi_{2} \in Sp(n', \bb{R}) \bigoplus Sp(n'', \bb{R}),$ we have 
$$\mu(\psi_{1} \oplus \psi_{2}) = \mu(\psi_{1}) + \mu(\psi_{2}).$$
\item [v.] (\textbf{Zero}) If a path has no eigenvalues on $S^1$, 
then its Conley-Zehnder index is 0.
\item [vi.] (\textbf{Signature})  Let $S$ be symmetric and nondegenerate with $||S|| < 2\pi$, and  let $\psi(t) = exp(JSt)$, then 
$$\mu_{CZ}(\psi) = \frac{1}{2}sign(S).$$  
\end{enumerate}
The Conley-Zehnder index is still insufficient for our purposes since we need the assumption that at time $T=1$ the symplectic matrix has no 
eigenvalue equal to 1.  We introduce yet another index for arbitrary paths.  We will call this index the Robbin-Salamon 
index and denote it $\mu_{RS}.$  

For this new index we simply add half of the signature of the crossing form at the terminal time of the path to the formula for the 
Conley-Zehnder index.
$$ \mu_{RS}(\Phi(t)) = \frac{1}{2}sign(\Gamma(0)) + \sum_{t\neq 0\,, \,t\, a\,\, crossing}sign(\Gamma(t)) + \frac{1}{2}sign(\Gamma(T))$$
This index satisfies the same axioms as $\mu_{CZ}$ as well as the new property of catenation.  
This means that the index of the catenation of paths is the sum of the indices.
\begin{enumerate}
\item[vii.] (\textbf{Catenation axiom}) Suppose that $\Phi_1, \Phi_2$ are two paths of symplectic matrices which satisfy 
$\Phi_1(T)= \Phi_2(0)$.  Then the new path $\Psi$ defined by concatenation of $\Phi_1$ with $\Phi_2$ has index $\mu(\Phi_1) + \mu(\Phi_2).$
\end{enumerate}  

Both the Conley-Zehnder and Robbin-Salamon indices arise from the Maslov index for a path of symplectic matrices as follows: first, let us assume that $H_1(M,\bbz)=0$ and consider a closed Reeb orbit $\grg$ together with an embedded Riemann surface $\Sigma \subset M$ such that $\partial \Sigma= \gamma.$ To find the relevant path of symplectic matrices with which to compute the Maslov 
index, one then pulls back the contact bundle $\cald$ to $\grS$, which then admits a trivialization, since it is a symplectic vector bundle over a Riemann surface with boundary. Then one considers the linearized Reeb flow about a Reeb orbit.  This linearized flow gives the desired path of symplectic matrices. As we shall see later when $M$ is the total space of an $S^1$ orbibundle over a simply connected symplectic orbifold, these indices are all essentially Chern numbers obtained by evaluating Chern classes of the orbifold strata on homology classes in $\calz$. It is important to understand that in a contact manifold, these indices depend on the choice of capping disk used to trivialize $\mathcal{D}.$ 
In particular if the closed Reeb orbit $\grg$ is contractible (which is always the case in this article), one trivializes $\cald$ by choosing a capping disk $\Sigma$ of $\grg$. If we consider another capping surface of the form $\Sigma^{'}=\Sigma \# S_A$ where $S_A$  represents a two dimensional homology class $A$ in $M$, then the Conley-Zehnder index of the orbit computed with $\Sigma^{'}$ will differ from that computed using $\grS$ 
by twice the first Chern class of $\mathcal{D}$ evaluated on $A$, namely
\begin{equation}\label{capdisk}
\mu_{CZ}(\grg;\grS_\grg\# S_A)= \mu_{CZ}(\grg;\grS_\grg) +2\langle c_1(\cald),A\rangle.
\end{equation}
Thus, the grading depends on the choice of trivialization. In order to address this dependence one considers the coefficients to be elements in the  Novikov ring. Give $H_2(M,\bbz)$ a grading $|\cdot|$  by setting $|A|= -2\langle c_1(\cald),A\rangle$ for any $A\in H_2(M,\bbz)$. Let $\calr$ be a submodule of $H_2(M,\bbz)$ with zero grading. Then the Novikov ring is the graded group ring $\bbq[H_2(M,\bbz)/\calr]$ whose element are formal power series of the form $\sum_iq_ie^{A_i}$ where $q_i\in \bbq$ and $A_i\in H_2(M,\bbz)/\calr$. Here as usual the notation $e^A$ is used to encode the multiplicative structure of a commutative ring with unit (cf. Chapter 11 of \cite{McSa04}).

Now that we have a grading we can define a graded chain complex $C_{*}$ generated by certain closed Reeb orbits with coefficients in the ring $\bbq[H_2(M,\bbz)/\calr]$. The grading in this chain complex is given 
by the Conley-Zehnder index shifted for convenience by $n-2$ and denoted by $|\grg|$. There some Reeb orbits for which the moduli space of 
holomorphic curves in $C(M)$ cannot be given a coherent orientation \cite{BoMo04}, so these `bad' Reeb orbits must be discarded. 
Let $\grg$ be a Reeb orbit with minimal period $T$, and $\grg_m$ be a Reeb orbit that covers $\grg$ with multiplicity $m$, so the period of $\grg_m$ 
is $mT$. The {\it bad} orbits are those for which the parity of the even multiples $|\grg_{2m}|$ disagrees with the parity for the odd multiples $|\grg_{2m-1}|$. A Reeb orbit that is not bad is said to be {\it good}.

\begin{definition}\label{chaincomp}
We define $C_*$ to be the graded chain complex freely generated by all good closed Reeb orbits with coefficients in the Novikov ring $\bbq[H_2(M,\bbz)/\calr]$.
\end{definition}

The differential $\d$ of this chain complex is given by an algebraic count 
of pseudoholomorphic curves in the symplectization $C(M)$ of $M$ which come in one dimensional families.  Explicitly, for $\gamma$ a good closed orbit of the Reeb vector field, $M$ simply connected, and $A$ a two dimensional homology class,  
the differential is given by the formula
\begin{equation}\label{hcdiff}
\partial \gamma =\sum_{A \in H_2(M, \bb{Z})} \sum_{\gamma'} \frac{1}{\kappa_{\gamma}} n_{\gamma, \gamma', A}e^{\bA}\grg',
\end{equation}
where $\bA$ denotes the image in $H_2(M,\bbz)/\calr$ of the homology class $A$, $\kappa_{\gamma}$ is the multiplicity of the Reeb orbit $\gamma$, $n_{\gamma, \gamma'}$ is the  algebraic 
count of elements in the moduli space $\oldmathcal{M}^A(\gamma, \gamma')$ of $J$-holomorphic curves into the symplectization of $M$ which are
asymptotically cylindrical over the closed Reeb orbits $\gamma$, $\gamma'$ representing the homology class $A.$ 
Note that $n_{\gamma, \gamma', A}$ is non-zero only if the dimension of this moduli space is $1.$ 
Unfortunately,  this does not always work, due to the lack of compactness of moduli spaces of pseudoholomorphic cylinders since, indeed the boundary of the 
compactification of this space can, in general, contain curves with more than two punctures. 
However, we can instead consider the supercommutative algebra generated by periodic orbits.  This means that instead of counting only cylinders, we now count
curves with an arbitrary number of negative punctures.      
This indeed gives a reasonable homology. The proofs that $\partial^2 =0$ and that the homology does not depend on 
choices of a contact form or an almost complex structure come from analysis of the boundary of moduli spaces of rigid curves and are discussed in ~\cite{ElGiHo00}.  These results depend on abstract transversality results for the $\delbar_J$ operator.  We will make the standing assumption that such transversality can be achieved, either by abstract perturbations or by the amenable geometry of the situation at hand.  The signs which appear in the algebraic count depend on coherent orientations of the moduli space are explained in ~\cite{BoMo04}. 
With this said we have
\begin{definition}\label{conhom}
The contact homology, denoted $HC$, is the homology of the complex $C_*$ of Definition \ref{chaincomp} with differential given by Equation (\ref{hcdiff}).
\end{definition}

In the above construction we need to make an assumption that the critical points of the action functional (\ref{eq:actfun}) are isolated in order to get a good index, 
i.e., we have to assume that that the Poincar\'{e} return map constructed about any periodic Reeb orbit has no eigenvalue equal to $1.$  This condition is generic; however, many natural contact forms, especially those which arise from
circle orbibundles are as far from generic as possible.  In order to calculate contact homology for such manifolds one must make
some sort of perturbation.  It is only in very nice situations that this is not extremely difficult.  The Morse-Bott version \cite{Bou02t,Bou03,ElGiHo00} allows us to use the symmetries of nice contact structures and symmetric almost complex structures, by rather than excluding 
non-isolated orbits, exploiting them.  This is accomplished by considering Morse theory on the quotient space, and relating critical points, and 
gradient trajectories of a Morse function to pseudoholomorphic curves in the symplectization of the contact manifold.  Since toric contact 
manifolds of Reeb type are always total spaces of circle orbibundles admitting Hamiltonian actions of tori and they admit nice Morse functions, the Morse-Bott formalism works quite well for us.  We follow a combination of ~\cite{ElGiHo00} and ~\cite{Bou02t} in what follows
applying the Morse-Bott set-up to our special case. 

\begin{defn} Let $(M, \mathcal{D})$ be a contact manifold with contact form $\eta.$  The \textbf{action spectrum}, 
$$\sigma(\eta) = \{r \in \bb{R} | r = \mathcal{A}(\gamma)\}$$ for $\gamma$ a periodic orbit of the Reeb vector field.
 
\end{defn}
\begin{defn} Let $T \in \sigma(\eta).$  Let $$N_{T} = \{ p \in M | \phi_{p}^{T} = p\},$$ $$S_{T} = N_{T}/S^{1},$$ where 
$S^1$ acts on $M$ via the Reeb flow.  Then $S_{T}$ is called the \textbf{orbit space} for period $T$.
 \end{defn}
When $M$ is the total space of an $S^1$-orbibundle the orbit spaces are precisely the orbifold strata. Let us give the definition that our contact form is of Morse-Bott type: 
\begin{defn}
A contact form, $\eta$ is said to be of \textbf{Morse-Bott type} if
\begin{enumerate}
\item[i.] The action spectrum 
$\sigma(\eta)$ is discrete.
\item[ii.] The sets $N_{T}$ are closed submanifolds of M, such that the rank of $d\eta|_{N_{T}}$ is locally constant and 
$$T_{p}(N_{T}) = ker(d\phi_{T} - I).$$
\end{enumerate}
\end{defn}

\begin{remark}These conditions are the Morse-Bott analogues for the functional on the loop space of $M.$
\end{remark}
Notice that in the case of $S^1$ orbibundles this is always satisfied.

Rather than set up Morse-Bott contact homology in full generality, let us do this for the special case of an $S^1$-orbibundle.  In this case
the contact form is of Morse-Bott type \cite{Bou02t}.  Let $T_1, \ldots, T_m$
be all possible simple periods for closed Reeb orbits.  Let $\phi_x^{t}$ denote the flow of the Reeb vector field. Let 
$$N_{T_j} = \{x \in M | \phi_x^{T_j} = x \},\qquad S_{T_j} = N_{T_j}/{S^1}.$$ 
For each $j$, we choose a Morse function $f_j$ on $S_{T_j}$ and, using appropriate bump functions build a Morse function, $f$ on all of 
$M$ which descends under the quotient by the Reeb action to each orbit space.
Now, we perturb $\eta$ by
\begin{equation} 
\eta_f = (1+ \epsilon f) \eta.  \label{eq:mbform}
 \end{equation}
For almost all $\epsilon$, the closed Reeb orbits of $\eta_f$  are isolated, and, for bounded action, they correspond to critical points of $f.$
Note that the Reeb orbits of $\eta$ within each stratum all have the same Robbin-Salamon index.  The following formula ~\cite{Bou02t,CiFlHoWy96} 
computes the Conley-Zehnder index of $\eta_f$ in terms of the Robbin-Salamon index of any Reeb orbit in a particular orbit space,
\begin{equation}
|\gamma|= \mu_{CZ}(\gamma) = \mu_{RS}(\gamma_T) - \frac{1}{2} dim (S_{T_j}) + ind_p(f_j).  \label{eq:gengrad}
\end{equation}
It is a theorem of Bourgeois ~\cite{Bou02t}, that we can compute the contact homology of this complex using so-called generalized holomorphic
curves, where we cannot determine exactly the asymptotics, only into which orbit space a particular curve intersects.  Using this it can be 
shown, up to transversality that we can compute this contact homology via the chain complex generated equivalently by closed Reeb orbits, as above or by 
the critical points of a Morse function on each stratum.  Then we define $HC_{*}(M, \mathcal{D})$ to be the homology of this chain complex with grading given by 
(\ref{eq:gengrad}), and the differential given by counting rigid generalized holomorphic curves, see ~\cite{Bou02t}.  Under suitable conditions $\partial ^2=0,$ 
and then this gives a well defined differential graded algebra.  Sometimes there are obstructions to invariance and $\partial^2 =0$, when this happens we must consider 
the graded \emph{algebra} generated by the above orbits or critical points, and instead of counting cylinders we count rigid genus $0$
curves with one positive puncture and arbitrarily many negative punctures.
We now have a contact form with isolated closed Reeb orbits.  Since these orbits correspond to critical points of a Morse function, we may think of
the generators of contact homology either as isolated orbits, or as critical points of a Morse function on each orbit space.

Let us now consider the Morse-Bott complex.  Since the closed Reeb orbits of action less than some $T_0$ all correspond to critical points of our Morse function, we 
consider the differential graded algebra with coefficients in $H_2(M_\bfp,\bb{Z})$ over the critical points of $|\mu_2|^2$ on each stratum of the orbit space.  This
is equivalent to considering the complex of Reeb orbits for a perturbed generic contact form with the same coefficient ring.  For any Reeb orbit in $S_T$ the grading is given by Equation (\ref{eq:gengrad}) with $f_j=|\mu_2|^2_T$ where $|\mu_2|^2_T$ denotes the restriction of $|\mu_2|^2$ to $S_T$. The coefficient ring encodes information about the homology classes of holomorphic 
curves which appear in the differential.  Multiplication by a ring element, $A$ decreases the grading by $2\langle c_1(\mathcal{D}_\bfp), A \rangle.$
Here one must use caution when computing the Robbin-Salamon index because $c_1(\mathcal{D})\neq 0.$ Since the computation depends explicitly on the choice of disk $\Sigma$ used to trivialize $\mathcal{D}$, we must address this dependence.

As we shall see later the index depends on certain spherical homology classes in $H_2(\calz, \bb{Z}).$  One possibility is to label Reeb orbits with 
the homology class corresponding to the choice of disk in $M$.  This amounts to labeling Reeb orbits with a subscript as in $\gamma_A$ or by considering
pairs $(\gamma, A).$  Some authors consider the Robbin-Salamon and Conley-Zehnder indices as taking \emph{two} arguments, namely a Reeb orbit
and a Riemann surface or homology class.  We will often suppress this notation in the sequel.   

Now we discuss further the relationship between holomorphic curves and Reeb orbits, which is, of course, of fundamental importance to the task at hand.  The proof of the following proposition appears in ~\cite{BEHWZ}.

\begin{prop} \label{ReebAsymp} 
Let $J$ be an almost complex structure on the symplectization of a contact manifold.  Suppose further
that the contact form is of Morse­-Bott type.  Let $u$ be a $J$-holomorphic curve with finite non-zero
energy. Then there exists a time greater than 0 and a periodic orbit $\gamma$ of the Reeb vector field such that $u$ is asymptotically cylindrical over $\gamma.$  Moreover the convergence is exponential.  
\end{prop}
     
This gives us the basic picture of holomorphic curves in this setting.  It remains to understand the 
topological structure of the space of such curves so that we may use them to define the differential in contact 
homology.  The first order of business is to understand the compactness properties of curves with prescribed 
positive and negative asymptotics.  Then we wish to understand the virtual dimensions of these spaces.

Although all of this can be defined in greater generality, we will restrict attention to the case of rational curves.  Given a pseudoholomorphic curve in the symplectization of
$M$ writing the component of the curve in the real direction $a(s, t).$  We call a puncture {\it positive} if $a$ is unbounded above, and {\it negative} if it is unbounded
below. 
By a version of the removable singularity theorem for pseudoholomorphic curves we know that curves with punctures whose cone component maps at a finite 
distance can be extended over the puncture. In this paper we consider only curves with a single positive puncture.  
Proposition ~\ref{ReebAsymp} tells us that around punctures, holomorphic maps look like cylinders over Reeb orbits near infinity.

\subsection{Moduli Spaces of Holomorphic Curves}
To define the differential we must consider moduli spaces of genus $0$ $J$-holomorphic curves into the symplectization of $M$.  
These curves can be considered as orbi-sections of the line orbi-bundle associated to $M$ as an $S^1$ orbi-bundle over $\calz$ with
its zero section removed.  We can, in principle, think of this differential as the standard one which comes with a generic choice of contact form.
In that case, one considers different moduli spaces than those which we are using here.  The fact that the homology is the same in both pictures 
comprises the main result of \cite{Bou02t}. 

\begin{definition}\label{modspace}
Let $(M, \eta)$ be a contact manifold such that the contact form $\eta$ is of Morse-Bott type.  Let $S^+,S_1, \ldots , S_m$ denote orbit spaces corresponding to the action by $R_{\eta}.$  Then we define the moduli space 
\begin{equation}
 \oldmathcal{M}_J(S^+,S_1, \ldots, S_k)
\end{equation}
to be the space of $J$-holomorphic punctured spheres into the symplectization of $M$ with $k+1$ punctures, 
which are cylindrically asymptotic to a Reeb orbit in $S^+$ at the positive puncture and to an orbit in $S_i$ at each negative puncture, $i=1,\ldots,k$. 
\end{definition}

These moduli spaces are not always compact, thus we must make sense of the compactification in order to understand 
how we should define the differential in contact homology.  We want to develop an analogue of Bott's version of 
Morse theory.  Recall that when critical points of a Morse-Bott function are not isolated, yet come in families, 
the differential in homology splits into an part which accounts for gradient trajectories connecting critical 
submanifolds, and those which connect critical points of a Morse function on each critical submanifold.  We think 
of the differential splitting into an external and an internal piece.  One can also define a similar differential 
on contact homology, but now we need a more general set up for the moduli spaces.  Our moduli spaces consist of 
maps into the symplectization of $M.$    The boundary of this moduli space consists of broken
holomorphic curves, i.e., buildings of curves with negative and positive asymptotics matching up at Reeb orbits.
As long as the contact form and almost complex structure is fixed we know that the compactification of 
$\oldmathcal{M}_J(S^+,S_1, \ldots, S_k)$ consists of ``broken trajectories,'' even in the Morse-Bott case.
To use the Morse-Bott setting to compute contact homology we need to perturb $\eta$ and $J.$  In his thesis
Bourgeois proved that when these structures are also allowed to vary and in particular when $\epsilon$ goes to zero in the perturbation of the
contact form, then one obtains broken curves, but now 
though parts of these curves converge to Reeb orbits, they no longer ``match up'' in the same way as before.  Now
the cylindrical parts are asymptotically cylindrical over Reeb orbits, but the lower and upper ends are connected
by cylinders over fragments of gradient trajectories of a Morse function on the orbit space containing the Reeb orbits over which the 
components are asymptotic.  These are fragments since we cannot control exactly where such a curve will approach the orbit space. 
One should think of this as analogous as the situation in ordinary Morse-Bott homology, 
where we have gradient trajectories connecting different critical submanifolds, and gradient trajectories of 
Morse functions within each critical submanifold.  Let us now state the definitions and theorems which make the 
above summary precise.  All of the following can be found in ~\cite{BEHWZ}, or ~\cite{Bou02t}.

These spaces are, first of all, defined in terms of \emph{stable nodal Riemann surfaces}.  Let $\Sigma$ denote a genus
$0$ Riemann surface which is possibly disconnected. Suppose further that we are given two disjoint finite collections of points $D$ and $V$, where $D$ has an even number of elements which we 
interpret and will write as pairs $(\ol{d_i}, \ul{d_i})\in D\times V\subset \grS$.  The set $D$ is the set of \emph{special} marked 
points.  The stability condition means that we require the union of the two sets of marked points to have 
cardinality at least $3$.  From such a collection we get a new (possibly singular) Riemann surface by
identifying the points of each pair $(\ol{d_i} , \ul{d_i})$.  Finally we let $Z$ denote a finite set which we
call \emph{punctures.}
\begin{definition}
A stable nodal Riemann surface is a quintuple \break $(\Sigma, j, V, Z, D)$, where the sets $V,Z$ and $D$ are defined in the 
preceding paragraph, and  $j$ is a complex structure on $\Sigma.$ 
\end{definition}

\begin{definition} 
 Let $(M, \eta)$ be a contact manifold, $W$ its symplectization, $J$ an almost complex structure
defined as in ~\ref{goodJ}.  A \textbf{nodal holomorphic curve}, or a \textbf{holomorphic building} of height $1$ is a proper 
holomorphic map of finite energy 
$$\tu =(a, u): \Sigma \setminus Z \ra{1.7} W= \bb{R} \times V$$ 
where $a : \Sigma \setminus Z \ra{1.5} \bb{R}$ and
$ u : \Sigma \setminus Z \ra{1.4} M$, and $\tu$ satisfies
$\tu(\ol{d_i}) =  \tu(\ul{d_i})$
for each pair $(\ol{d_i}, \ul{d_i})\in D\times V$.
\end{definition}

We require the stability condition that at least one component of the image of such a curve is different from
a cylinder over a Reeb orbit, and that for each constant component of such a map, that the associated component
of $\Sigma$ is stable.  Moreover the set $Z$ is divided into two disjoint sets of positive and negative punctures.  

The space of holomorphic buildings of height $1$ is not compact.  Therefore we need to define the space 
of height $k$ holomorphic buildings.  Suppose that we are given $k$ possibly disconnected height $1$ holomorphic
buildings
$$\tu_m = (a_m, u_m, M_m, Z_m, D_m), \qquad m=1,\cdots,k.$$ 
We can then consider the nodal Riemann surface obtained by matching Reeb orbits corresponding to boundary circles
around the negative punctures of $Z_m$ to the positive punctures of $Z_{m-1}$, then we know that the positive 
asymptotics of $u_{m-1}$ match up with the asymptotics of the negative puncture of $u_m$.  Given a cross ordering 
on the sets $M­­_j$ along with local orientation reversing diffeomorphisms from neighborhoods of positive puncture of $\tu_{m-1}$ to
neighborhoods of the negative punctures of $\tu_m,$ we obtain what is called a \emph{holomorphic building of height
k}.
   
The main compactness theorem for symplectizations is (\cite{BEHWZ}, Section 10)
\begin{theorem}
 Given a sequence of stable holomorphic buildings of height $1$ there is a subsequence converging to 
a stable curve of height $k.$ 
\end{theorem}

This theorem is valid as long as the contact form, and almost complex structure remain fixed.  However, to really
use the full power of the Morse-Bott set up, one needs to see what happens when a sequence is given where all given
structures vary.  In other words we have a sequence of curves, which have non-degenerate asymptotics (i.e., the closed Reeb orbits are isolated), where the 
contact form in the limit has degenerate asymptotics (i.e., the closed Reeb orbits come in families). The question is what does this compactification look like?

To describe the limit curves we must first look at the manifolds of Reeb orbits coming from our contact form 
of Morse-Bott type.  Let $T$ be some period of the Reeb vector field, and let $N_T$ denote the manifold foliated
by orbits of period $T.$  We choose a Morse function, $f$, on $N_T$ invariant with respect to the Reeb flow, and constant
in the normal direction to $N_T$.  Now for $\epsilon > 0 $ sufficiently small we get the new contact form 
$$\eta_{\epsilon} = (1 + \epsilon f) \eta.$$ 
We now get corresponding symplectic forms on $W$, and new almost complex structures $J_{\epsilon}$ defined by 
$$J_{\epsilon}|_{\mathcal{D}} = J$$
and 
$$J_{\epsilon} \frac{\partial}{\partial s} = R_{\epsilon}.$$
Here $R_{\epsilon}$ denotes the Reeb vector field of $\eta­_{\epsilon}.$ 
We want to see what the limits of height one curves look like when $\epsilon \rightarrow 0$.

\begin{definition} \label{defgenholcurve}
 A \textbf{generalized holomorphic curve with $l$ sublevels} consists of a collection of the following data:
 $l$ holomorphic buildings of height $1$ satisfying the conditions for a holomorphic building of height $l$ with
the added condition that there are $l +1$ collections of cylindrical gradient trajectories of Morse functions corresponding to the 
asymptotics of each holomorphic building of height $1$.
 \end{definition}

Indeed it was proved in ~\cite{Bou02t} that the boundary of $\oldmathcal{M}_J(S^+,S_1, \ldots, S_k)$ consists exactly of broken curves with Morse
trajectories as in definition \ref{defgenholcurve}. Thus, we have
\begin{theorem} \label{compGenHol}
 A sequence of holomorphic curves defined as in the previous paragraph, has a subsequence converging to 
a generalized holomorphic curve of height $l$ as $\epsilon \rightarrow 0.$
\end{theorem}

On the other hand, we really want to know that these generalized holomorphic curves tell the whole story, i.e., we would like to know
that these are enough to make our computations later.  More precisely, we would like to know if, given a rigid generalized holomorphic curve for
an almost complex structure determined by our contact form of Morse-Bott type, that there is a sequence of $J_{\epsilon}$-holomorphic curves converging
to it.  To do this one must construct a family of approximately holomorphic curves obtained by gluing together fragments of holomorphic cylinders, and 
then finding nearby genuinely holomorphic cylinders. Then one uses Floer's Picard lemma to prove that such a map must be surjective.  A proof is
sketched in ~\cite{Bou02t}.

\subsection{The differential}
We define for $p \in CH_*$ the differential on generators
\begin{equation}\label{conhomdiff}
d p = \partial_{MSW}p + d_{CH} p
\end{equation}
where $\partial_{MSW}$ is the differential on the Morse-Smale-Witten complex determined by our choice of Morse function.  $d_{CH}$ is defined by 
$$d_{CH} p = \sum_{A\in H_2(\calz, \bb{Z})}\sum_{dim \oldmathcal{M}(S^{+}, \ldots, S_k)/\bb{R}=0}\frac{1}{\kappa_p}\#^{alg}\oldmathcal{M}(S^{+}, \ldots, S_k)/\bb{R} q_1 
\cdot \ldots \cdot q_n e^{A} .$$ 
This is to be extended to the whole algebra by linearity and the graded Leibniz rule.
Here $\kappa_p$ denoted the multiplicity of the Reeb orbit associated to $p.$ The $q_j$'s denote critical points of a Morse function, $f$,
on each orbit space $S_{T_j},$ or alternatively closed Reeb orbits of a perturbed contact form $\eta_{f}.$  $A$ denotes the homology classes 
of a pseudoholomorphic curve asymptotically cylindrical over the specified Reeb orbits associated to each critical point.  
In other words, $d_{CH}$ counts rigid generalized pseudoholomorphic curves ~\cite{Bou02t} with positive asymptotics in $S^{+}$ and negative 
asymptotics in $S_1, \ldots S_k.$  There is a sign as in the non-degenerate case as defined in ~\cite{BoMo04}.  
This will not matter much to us since in our cases these signs never get a chance to show up.  Notice that $d_{CH}$ vanishes as long as the relevant
moduli spaces have dimension greater than $1$ before the $\bb{R}$ quotient by translation.  As mentioned before, $e^A$ is a convenient notation
for keeping track of the homology class of a rigid $J$-holomorphic curve.  Here we have suppressed the choice of 
Riemann surface trivializing an orbit corresponding to $p.$

\begin{prop}\label{diffvanish}
 When $M$ is the total space of an orbibundle of a symplectic orbifold, then there are no rigid holomorphic curves into the symplectization of $M$.
\end{prop}
\begin{proof}
There  is an effective $\bb{R}$-action, as well as that of a circle on the moduli spaces; hence, these spaces have dimension at least $2.$ So they can never be rigid.    
\end{proof}

\begin{remark}[Note on transversality].
Though in some special cases we can use the nice properties of toric manifolds to determine regularity of the moduli spaces of curves defined above, 
the proofs both of invariance of contact homology, as well as to make the proof that the Morse-Bott complex actually computes the homology of 
the perturbed complex requires the use of abstract perturbations of the $\delbar_J$ operator.  The authors believe that the results of Hofer, Wysocki, and Zehnders polyfold theory shall provide a good framework for this problem, however,  we make it a standing assumption that there exists an abstract perturbation of the $\delbar_{J}$ operator which 
makes its linearization surjective.   
Proposition \ref{diffvanish} requires no transversality result for the $\delbar_J$ operator since we can get at least these two dimensions without any appeal to abstract Fredholm theory.  This does not make the transversality problem go away, however, since it is still needed in proofs of invariance, and independence of choices.  Moreover, when one wishes to analyze higher dimensional moduli spaces by adding marked points, one needs the relevant dimension formulae to hold, although this can be handled in many cases using the fact that $J$ can be chosen integrable in these toric situations.  
We should also mention that, even without the transversality assumption mentioned here, we can obtain a weaker version of invariance as we shall see later.
\end{remark}

\section{Toric Contact Structures of Reeb Type}

It is well understood \cite{BM93,Ler02a} that toric contact structures on manifolds of dimension greater than three come in two types, those where the action of the torus is free, and those where it is not. The latter contain an important special subclass known as toric contact structures of Reeb type \cite{BG00b}, and the former cannot occur on $S^3$-bundles over $S^2$ so we shall not consider them further in this paper. Moreover, toric contact structures of Reeb type are always quasi-regular, and correspond to convex polyhedral cones in the dual of the Lie algebra of the torus \cite{BG00b,Ler02a}.

\begin{definition}
A {\bf toric contact manifold} $(M,\cald,{\oldmathcal A})$ is a contact manifold of dimension $2n+1$ together with an effective action of a torus $T$ of dimension $n+1$ that leaves the contact structure invariant, i.e. if ${\oldmathcal A}:T\times M\ra{1.6} M$ denotes the action map then ${\oldmathcal A}_*\cald=\cald$.
\end{definition}

By averaging over $T$ we can always find a contact 1-form $\eta$ representing $\cald$ such that ${\oldmathcal A}^*\eta=\eta$. In this case we also have ${\oldmathcal A}_*R=R$ for the Reeb vector field. A toric contact manifold is said to be of {\it Reeb type} if there is a contact form $\eta\in \gC^+(\cald)$ whose Reeb vector field lies in the Lie algebra $\gt$ of $T$.

\begin{definition}\label{torconequiv}
Two toric contact manifolds $(M,\cald,{\oldmathcal A})$ and $(M',\cald',{\oldmathcal A}')$ are said to be {\bf equivariantly equivalent} (or equivalent toric contact manifolds) if there exists a diffeomorphism $\varphi:M\ra{1.6} M'$ such that $\varphi_*\cald=\cald'$ and $\varphi\circ {\oldmathcal A}= 
{\oldmathcal A}'\circ \varphi$.
\end{definition}

Toric contact manifolds were classified by Lerman \cite{Ler02a}. Notice that the first condition of Definition \ref{torconequiv} says that the diffeomorphism $\varphi$ is a contactomorphism, while the second condition says that ${\oldmathcal A}'$ is conjugate to ${\oldmathcal A}$ under $\varphi$. So if the second condition fails the tori generated by ${\oldmathcal A}$ and $\varphi^{-1}\circ {\oldmathcal A}'\circ \varphi$ belong to distinct conjugacy classes in the contactomorphism group $\gC\go\gn(M,\cald)$. Furthermore, to each such conjugacy class there is an associated toric CR structure $(\cald,J)$ which by Theorem 7.6 of \cite{Boy10a} is unique up to biholomorphism.

\subsection{Contact Reduction} It is well known (cf. \cite{BG00b,Ler02a}) that every contact toric structure of Reeb type can be obtained by symmetry reduction of the standard sphere by a compact Abelian group $T$, and that this is equivalent to the symplectic reduction of the standard symplectic structure on $\bbc^N\setminus \{0\}$ by a compact Abelian group which commutes with the action of dilations of the cone. For this one must choose the zero level set of the toral moment map. This equivalence can be described  by the commutative diagram (cf. \cite{BG05}, pg 293): 
\begin{equation}\label{commquot}
\begin{matrix} S^{2N-1}_\bfw &\longleftrightarrow &\bbc^{N}\setminus \{0\} \\
                         \Downarrow &&           \Downarrow \\
                         M^{2n-1} &\longleftrightarrow &C(M^{2n-1}),
\end{matrix}
\end{equation}
with $\dim T=N-n$.
According to Lerman \cite{Ler04} $\pi_1(M)=\pi_0(T)$ and $\pi_2(M)=\pi_1(T)=\bbz^{\dim T}$. This implies

\begin{lemma}\label{circlered} Let $M$ be an $S^3$-bundle over $S^2$.
Every toric contact structure on $M$ can be obtained by contact circle reduction of the standard contact structure on $S^7$. 
\end{lemma}

We now describe this reduction. First the standard $T^4$ action on $\bbc^4$ is $z_j\mapsto e^{i\theta_j}z_j$, and its moment map $\Upsilon_4:\bbc^4\setminus \{0\}\ra{1.7} \gt^*_4=\bbr^4$
is given by 
\begin{equation}\label{sph4mom}
\Upsilon_4(z)=(|z_1|^2,|z_2|^2,|z_3|^2,|z_4|^2). 
\end{equation}
Now we consider the circle group $T(\bfp)$ acting on $\bbc^4\setminus \{0\}$ by
\begin{equation}\label{s1action}
(z_1,z_2,z_3,z_4)\mapsto (e^{ip_1\theta}z_1,e^{ip_2\theta}z_2,e^{-ip_3\theta}z_3, e^{-ip_4\theta}z_4),
\end{equation}
where $\bfp$ denotes the quadruple $(p_1,p_2,p_3,p_4)$ with $p_i\in\bbz^+$ and we assume $\gcd(p_1,p_2,p_3,p_4)=1$. We have an exact sequence of commutative Lie algebras 
\begin{equation}\label{Liealgseq}
0\ra{1.5}\gt_1(\bfp)\ra{1.5}\bbr^4\fract{\tvarpi}{\ra{1.5}}\gt_3(\bfp)\ra{1.5} 0
\end{equation}
where $\gt_1(\bfp)$ is the Lie algebra of $T(\bfp)$ generated by the vector field $L_\bfp =p_1H_1+p_2H_2-p_3H_3-p_4H_4$.

Dualizing (\ref{Liealgseq}) gives
\begin{equation}\label{g*exactseq}
0\ra{1.5}\gt_3^*(\bfp)\fract{\tvarpi^*}{\ra{1.5}}(\bbr^4)^*\ra{1.5}\gt_1^*(\bfp)\ra{1.5} 0. 
\end{equation}
The moment map $\Upsilon_1:\bbc^4\setminus \{0\}\ra{1.6} \gt_1^*=\bbr$ for this action is given by
\begin{equation}\label{conemomentmap}
\Upsilon_1(z)=p_1|z_1|^2+p_2|z_2|^2-p_3|z_3|^2-p_4|z_4|^2. 
\end{equation}

Now consider the 1-form 
\begin{equation}\label{sph1form}
\eta_0=-\frac{i}{2}\sum_{j=0}^n(z_jd\bar{z}_j-\bar{z}_jdz_j).
\end{equation}
on $\bbc^4\setminus \{0\}$ together with the vector field 
\begin{equation}\label{Reebsph}
R_\bfp=\sum_jp_jH_j
\end{equation}
where $H_j=-i(z_j\frac{\partial}{\partial
z_j}-\bar{z}_j\frac{\partial}{\partial \bar{z}_j})$. Imposing the constraint $\eta_0(R_\bfp)=1$ gives $S^7$ represented as $\sum_jp_j|z_j|^2=1$. Then $\eta_0$ pulls back to a contact form on $S^7$, also denoted by $\eta_0$, with Reeb vector field $R_\bfp= p_1H_1+p_2H_2+p_3H_3+p_4H_4$. By a change of coordinates one easily sees that this represents the standard contact structure on $S^7$.

So the zero level set $\Upsilon^{-1}_1(0)$ is diffeomorphic to a cone over $S^3\times S^3$, or equivalently restricting to $S^7$, the zero level set of $\mu_{\eta_0}$ is $S^3\times S^3$, represented by
\begin{equation}\label{s3s3}
p_1|z_1|^2+p_2|z_2|^2=\frac{1}{2}, \qquad p_3|z_3|^2+p_4|z_4|^2=\frac{1}{2}.
\end{equation}

The action of $T(\bfp)$ is free on this zero set if and only if $\gcd(p_i,p_j)=1$ for $i=1,2$ and $j=3,4$. So assuming these gcd conditions our reduced contact manifold is the $M_\bfp=(S^3\times S^3)/T(\bfp)$ whose contact form is the unique 1-form $\eta_\bfp$ satisfying $\gri^*\eta_0=\grr^*\eta_\bfp$ where $\gri:\mu_{\eta_0}^{-1}(0)\ra{1.6} S^7$ and $\grr:\mu_{\eta_0}^{-1}(0)\ra{1.6} M_\bfp$ are the natural inclusion and projection, respectively. In order to identify $M_\bfp$ we consider the $T^2(\bfp)$ action on $\mu_{\eta_0}^{-1}(0)\approx S^3\times S^3$ generated by the $S^1$ action (\ref{s1action}) together with the $S^1$ action generated by the Reeb vector field $R_\bfp$. We have

\begin{definition}\label{admis}
We say that the quadruple $\bfp=(p_1,p_2,p_3,p_4)$ of positive integers is {\bf admissible} if $\gcd(p_i,p_j)=1$ for $i=1,2$ and $j=3,4$. We denote the set of admissible quadruples by $\cala$.
\end{definition}

Let us describe some obvious equivalences. We can interchange the coordinates $z_1\leftrightarrow z_2$, likewise $z_3\leftrightarrow z_4$. Thus, without loss of generality we can assume that $p_1\leq p_2$ and $p_3\leq p_4$. We can also interchange the pairs $(z_1,z_2)$ and $(z_3,z_4)$. 

We are now ready for

\begin{lemma}\label{t2quot}
Let $\bfp$ be admissible. Then quotient space of $\mu_{\eta_0}^{-1}(0)\approx S^3\times S^3$ by the $T^2(\bfp)$ action is the orbifold $\bbc\bbp(\bp_1,\bp_2)\times \bbc\bbp(\bp_3,\bp_4)$ where $(p_1,p_2)=k(\bp_1,\bp_2)$ and $(p_3,p_4)=l(\bp_3,\bp_4)$ with $\gcd(\bp_1,\bp_2)=\gcd(\bp_3,\bp_4)=1$. Moreover, the cohomology class in $H^2_{orb}(\bbc\bbp(\bp_1,\bp_2)\times \bbc\bbp(\bp_3,\bp_4),\bbz)$ of this orbibundle is the class of the K\"ahler form $\gro_\bfp=l\gro_{\bp_1,\bp_2}+k\gro_{\bp_3,\bp_4}$.
\end{lemma}

\begin{proof}
The $T^2(\bfp)$ action on $S^3\times S^3$ splits as a weighted $S^1$ action on each factor. Setting $k=\gcd(p_1,p_2)$ and $l=\gcd(p_3,p_4)$ we see after reparameterizing that the quotient of the first factor is $\bbc\bbp(\bp_1,\bp_2)$, and similarly for the second factor. Note that $\bfp\in \cala$ implies $\gcd(k,l)=1.$

We have an exact sequence of groups 
$$0\ra{1.5}T(\bfp)\ra{1.5}T^2(\bfp)\ra{1.5} S^1(R_\bfp)\ra{1.5} 0$$
where $S^1(R_\bfp)$ is the circle generated by the Reeb vector field $R_\bfp$.
Thus, we have the commutative diagram
\begin{equation}\label{orbibranchcover2}
\xymatrix{ S^3\times S^3\ar[dd]^{\pi_1\times \pi_2}\ar[rd]^\grr & {}\\
&M_\bfp\, .\ar[ld]^\pi\\
\bbc\bbp(\bp_1,\bp_2)\times \bbc\bbp(\bp_3,\bp_4)&\\
 }\end{equation}
We want to determine the integral orbifold first Chern class (Euler class) of the $S^1$ orbibundle given by the southwest arrow. That is, we look for the class $a\gra+b\grb\in H^2_{orb}(\bbc\bbp(\bp_1,\bp_2)\times \bbc\bbp(\bp_3,\bp_4),\bbz)$ which transcends to the zero class on $M_\bfp$ where $\gra$ and $\grb$ are primitive classes in each factor. (See Chapter 4 of \cite{BG05} for a discussion of these orbifold classes.) For this we take the Bochner-flat K\"ahler metrics on the weighted projective spaces as described in \cite{Bry01,DaGa06,Gau09} whose area is $2$. We denote the corresponding K\"ahler forms by $\gro_{\bp_1,\bp_2}$ and $\gro_{\bp_3,\bp_4}$, so that $\gra=[\gro_{\bp_1,\bp_2}]$, and $\grb=[\gro_{\bp_3,\bp_4}]$. Now according to the action (\ref{s1action}) the circle wraps around $k$ times on the first factor and $l$ times with the reverse orientation on the second. So if we take the K\"ahler form to be
\begin{equation}\label{grocpcp}
\gro_\bfp=l\gro_{\bp_1,\bp_2}+k\gro_{\bp_3,\bp_4},
\end{equation}
its class pulls back to zero under $\pi$, since $\pi^*[\gro_{\bp_1,\bp_2}]=k\grg$ and $\pi^*[\gro_{\bp_3,\bp_4}]=-l\grg$ where $\grg$ is a generator of $H_2(M_\bfp,\bbz)\approx \bbz$. \end{proof}

Under these conditions we have

\begin{theorem}\label{quotthm}
$M_\bfp$ is diffeomorphic to $S^2\times S^3$ if $p_1+p_2-p_3-p_4$ is even, and diffeomorphic to $X_\infty$, the non-trivial $S^3$-bundle over $S^2$, if $p_1+p_2-p_3-p_4$ is odd.
\end{theorem}

\begin{proof} 
We know from the reduction procedure and Lemma \ref{circlered} that $M_\bfp$ is simply connected and $\pi_2(M_\bfp)=\bbz$. So by the Barden-Smale classification of simply connected 5-manifolds $M_\bfp$ is determined by its second Stiefel-Whitney class $w_2(M)$. Moreover, since $TM_\bfp$ splits as $\cald_\bfp$ plus a trivial line bundle, $w_2(M)$ is 
the mod 2 reduction of $c_1(\cald)$. So the theorem will follow immediately from the following lemma. 
\end{proof}

\begin{lemma}\label{c_1D}
The first Chern class of the contact bundle $\cald_\bfp=\ker\eta_\bfp$ on $M_\bfp$ is given by
$$c_1(\cald_\bfp)=(p_1+p_2-p_3-p_4)\grg$$ 
where $\grg$ is the positive generator of $H^2(M_\bfp,\bbz)\approx \bbz$.
\end{lemma}

\begin{proof}
We begin by computing the orbifold first Chern class of 
$$\bbc\bbp(\bp_1,\bp_2)\times \bbc\bbp(\bp_3,\bp_4).$$ 
From \cite{BG05} we see that $p^*c_1^{orb}$ is given by 
\begin{equation}\label{c1cpcp}
(\bp_1+\bp_2)[\gro_{\bp_1,\bp_2}]+(\bp_3+\bp_4)[\gro_{\bp_3,\bp_4}]\in H^2_{orb}(\bbc\bbp(\bp_1,\bp_2)\times \bbc\bbp(\bp_3,\bp_4),\bbz) 
\end{equation}
which pulls back to the basic first Chern class $c_1(\calf_{R_\bfp})$ in the basic cohomology group $H^2_B(\calf_{R_\bfp})$ under the natural projection $\pi:M_\bfp\ra{1.6} \bbc\bbp(\bp_1,\bp_2)\times \bbc\bbp(\bp_3,\bp_4)$ by the circle action of $R_\bfp$.
Now we have an exact sequence \cite{BG05}
$$
\begin{matrix} &0& \\
                     &\decdnar{}& \\
          &H^{2}(M_\bfp,\bbz)&\\
          &\decdnar{}&\\
          0\ra{1.2} H^{0}_B(\calf_{R_\bfp})\fract{\grd}{\ra{1.5}}
H^{2}_B(\calf_{R_\bfp})\fract{\gri_*}{\ra{1.8}}&H^{2}(M_\bfp,\bbr)&\ra{1.3}\cdots
\end{matrix}
$$
with $\gri_*c_1(\calf_{R_\bfp})=c_1(\cald_\bfp)_\bbr$ and $\grd a=a[d\eta_\bfp]_B$. So $c_1(\cald_\bfp)_\bbr$ is $c_1(\calf_{R_\bfp})$ mod $[d\eta_\bfp]_B$ where $\eta_\bfp$ is the contact form on $M_\bfp$. Now since $\pi^*\gro_\bfp=d\eta_\bfp$, we know from the proof of Lemma \ref{t2quot} that $\pi^*[\gro_{\bp_1,\bp_2}]=k\grg$ and $\pi^*[\gro_{\bp_3,\bp_4}]=-l\grg$ holds over $\bbz$. Thus, since $\pi_1(M_\bfp)=\{\BOne\}$ we have over $\bbz$
\begin{flalign*}
c_1(\cald_\bfp) &=(\bp_1+\bp_2)\pi^*[\gro_{\bp_1,\bp_2}]+ (\bp_3+\bp_4)\pi^*[\gro_{\bp_3,\bp_4}] \\
&= k(\bp_1+\bp_2)\grg-l(\bp_3+\bp_4)\grg=(p_1+p_2-p_3-p_4)\grg.
\end{flalign*}
This completes the proof of the lemma.
\end{proof}

\begin{corollary}\label{s3s2cor}
Every toric contact structure on an $S^3$-bundle over $S^2$ can be realized as an orbifold fibration over $\bbc\bbp(\bp_1,\bp_2)\times \bbc\bbp(\bp_3,\bp_4)$ for some quadruple of positive integers $(p_1,p_2,p_3,p_4)$ satisfying $\gcd(p_i,p_j)=1$ for $i=1,2$ and $j=3,4$ and $(p_1,p_2,p_3,p_4)=(k\bp_1,k\bp_2,l\bp_3,l\bp_4)$. Thus, every toric contact structure on an $S^3$-bundle over $S^2$ has a positive Sasakian structure in its Sasaki cone.
\end{corollary} 

\begin{proof}
The first statement follows by combining Lemma \ref{circlered} with the proof above, and the second statement follows since the base orbifold $\bbc\bbp(\bp_1,\bp_2)\times \bbc\bbp(\bp_3,\bp_4)$ is log del Pezzo.
\end{proof}

We also have a characterization of the Sasaki cone. 
\begin{lemma}\label{Sasconered}
The (unreduced) Sasaki cone $\gt_3^+(\bfp)$ is given by
$$\gt_3^+(\bfp)=\{R\in \gt_3(\bfp)~|~ R=\sum_{i=1}^4a_i\varpi(H_i)~|~p_ia_j+p_ja_i>0~\forall i=1,2;~j=3,4\}.$$
\end{lemma}

\begin{proof}
The Sasaki cone is determined by the positivity condition $\eta_0(R)>0$ restricted to the zero level set of the $T(\bfp)$ moment map. Using Equation (\ref{s3s3}) this gives 
$$p_4(p_2a_1-a_2p_1)|z_1|^2+p_2(p_4a_3-a_4p_3)|z_3|^2+\frac{a_2p_4+a_4p_2}{2}>0$$
from which we obtain the result.
\end{proof}

It is easy to see that the argument in \cite{Ler03b} can be generalized to give
\begin{proposition}\label{Dpc1}
As a complex vector bundle $\cald_\bfp$ is determined uniquely by $p_1+p_2-p_3-p_4$.
\end{proposition}

Let $A(M_\bfp)$ denote the set of homotopy classes of almost contact structures on $M_\bfp$. This corresponds precisely with the homotopy classes of almost complex structures on the cone $C(M_\bfp)$ \cite{Sat77}, so $c_1(\cald)$ is an invariant of the homotopy class of almost contact structures. 

\subsection{Contact Homology for toric contact $5$-manifolds}

Let us consider the differential graded algebra (DGA) discussed above.  We start with the set of critical points of a 
Morse function as picked earlier.  Since we are working with toric manifolds of Reeb type in dimension $5$ we actually know that the fixed points of 
the $T^3$-action are isolated, hence the norm squared of the symplectic moment map on $\calz$ is a perfect Morse function.

We are interested in the orbit structure of the $T^3(\bfp)$ action on $M_\bfp$.

\begin{lemma}\label{fixReeb}
Consider the toric contact structure $\cald_\bfp$ on $M_\bfp$, an $S^3$-bundle over $S^2$. Then there are exactly four 1-dimensional orbits under the action of $T^3(\bfp)$, and they are closed. Moreover, these four orbits are Reeb orbits for all Reeb fields in the Sasaki cone $\gt_3^+(\bfp)$, as well as for a Reeb vector field in $\gc\go\gn(M_\bfp,\eta_\bfp)$ that is arbitrarily close to one in the Sasaki cone. Moreover, for a generic such Reeb vector field these are the only closed orbits.
 \end{lemma}

\begin{proof}
We have an exact sequence of groups,
$$\{0\}\ra{1.8} T(\bfp)\ra{1.8}T^4\ra{1.8} T^3(\bfp)\ra{1.8} \{0\},$$
and we consider the action of $T^3(\bfp)$ on the level set given by Equation (\ref{s3s3}) thought of as $T^4/T(\bfp)$. If $z_1\neq 0$, then we can choose $\theta =\theta_1$ of the standard $T^4$ angles. The remaining $T^3$ orbit will be 1-dimensional only if $z_2=0$ and one of the $z_3$ or $z_4$ is zero. This gives two closed $S^1$-orbits. On the other hand if $z_1=0$, then we must have $z_2\neq 0$, so we choose $\theta=\theta_2$, and as above this gives exactly the two close orbits with either $z_3$ or $z_4$ vanishing. Clearly, any Reeb vector field in $\gt_3^+$ leaves these Reeb orbits invariant, and since every Reeb vector field in the Sasaki cone is arbitrarily close to a quasi-regular one, the last statement follows from a result of Bourgeois \cite{Bou02t}.
\end{proof}

Let $\gt_2(\bfp)$ denote the Lie algebra of $T^2(\bfp)$. It is generated by the two vector fields $L_\bfp,R_\bfp$. We have an exact sequence of Lie algebras
\begin{equation}\label{laexact}
\{0\}\ra{1.8} \gt_2(\bfp)\ra{1.8} \gt_4\fract{\grr}{\ra{1.8}} \gg_2(\bfp)\ra{1.8} \{0\},
\end{equation}
where $\gg_2(\bfp)$ is generated by the vector fields $\bar{H}_1=\grr(H_1),\bar{H}_3=\grr(H_3)$.
We have a toric symplectic orbifold 
\begin{equation}\label{symtororb}
(\bbc\bbp(\bp_1,\bp_2)\times \bbc\bbp(\bp_3,\bp_4),\gro_\bfp)
\end{equation}
where the symplectic form is given by Equation (\ref{grocpcp}), and the torus $\gG_2(\bfp)$ is generated by the Lie algebra $\gg_2(\bfp)$. The moment map $\mu_2:\bbc\bbp(\bp_1,\bp_2)\times \bbc\bbp(\bp_3,\bp_4) \ra{1.6} \gg_2(\bfp)^*$ is given by $\mu_2(\bfz)=(|z_1|^2,|z_3|^2)$.

\begin{proposition}\label{momMorse}
The function $f=|\mu_2|^2$ is a perfect Morse function on the quotient $M_\bfp/S^1\approx\bbc\bbp(\bp_1,\bp_2)\times \bbc\bbp(\bp_3,\bp_4)$ whose critical points are precisely the four Reeb orbits of Lemma \ref{fixReeb}.
\end{proposition}

\begin{proof}
Since the critical points are isolated $f$ is a Morse function, and Morse-Bott functions that are the norm squared of a moment map are perfect \cite{LeTo97}.
It is easy to check directly (see also Lemma 3.1 of \cite{Kir84}) using the relations 
$$H_1\equiv aH_2\mod{\gg_2(\bfp)}, \quad H_3\equiv bH_4\mod{\gg_2(\bfp)}$$
for some $a,b\in\bbr$ that $f$ has precisely the four critical points 
$$[1,0]\times [1,0],~[1,0]\times [0,1],~[0,1]\times [1,0],~[0,1]\times [0,1]$$
and these correspond to the four Reeb orbits of Lemma \ref{fixReeb}.
\end{proof}

\begin{thm}\label{conhomorbibun}
In the case of circle reductions in dimension $5$, which have four Reeb orbits fixed by the $T^3$-action, 
the differential in Morse-Bott contact homology vanishes.
Moreover, the contact homology is determined completely by the homology of the base orbifold.  More precisely, it is given by the the homology groups of each orbits space.  The degree of each generator is given by the degree of the corresponding generator in homology, with a degree shift given by the Robbin-Salamon index.
\end{thm}
\begin{proof} 
By Proposition \ref{diffvanish} there are no rigid holomorphic curves. So $d_{CH}$ vanishes. But also since $|\mu|^2$ is a perfect $S^1$-invariant Morse function, the Morse-Smale-Witten 
differential $\d_{MSW}$ vanishes as well. Thus, the full differential (\ref{conhomdiff}) vanishes.  For each period of the Reeb flow,
we get a new Reeb orbit, corresponding to  some critical point of $f.$  Since $f$ is perfect, these critical points correspond not just 
to chains but to actual homology classes.  The statement about the grading follows from \ref{eq:gengrad}.
\end{proof}

The next proposition, though not a general proof of invariance of contact homology, does tell us that we do get an invariant in the world of $S^1$-orbibundles, whose bases admit a perfect Morse function.
\begin{proposition}
Let $M$ be a quasiregular contact manifold, such that its quotient by the Reeb vector field is a symplectic orbifold which admits a perfect Morse function.  Then if $M'$ is contactomorphic to $M$, is quasiregular, and the quotient by its Reeb vector field is also a symplectic orbifold which admits a perfect Morse function, then the two contact homology algebras are isomorphic. 
\end{proposition}

\begin{proof}
The conditions on $M$, $M'$, and their bases ensure that all parts of the differential vanish.  Therefore we may construct a map between these two algebras as in \cite{ElGiHo00} counting rigid curves in a symplectic cobordism between $M,$ $M'$.  The main difficulty is in seeing that this map is a chain map.  However, since the differentials vanish on both ends, the map is trivially a chain map, hence the two contact homology algebras are isomorphic.
\end{proof}

To compute the grading on contact homology it is useful to consider a special case of the \textbf{join} construction \cite{BGO06}.
 Since we can view our toric sphere bundles as quotients of $S^3 \times S^3$ we have a convenient way to compute indices.  This is of particular interest for strata of positive codimension,
since the orbits in the codimension 0 stratum behave exactly as in the regular case.  To define the 
join construction we start with two quasiregular contact manifolds, $M_1, M_2$ with contact forms $\eta_1, \eta_2,$ and bases 
$\calz_1, \calz_2$ with symplectic forms $\omega_1$ and $\omega_2$.  Then the product $M_1 \times M_2$ is a $T^2$-bundle over $\calz_1 \times \calz_2$.  We take the quotient of 
$M_1 \times M_2$ by the action of the circle obtained by gluing together Reeb orbits on each piece i.e., 
\begin{equation}
(z, w) \mapsto (e^{i k_1 \theta}z,e^{-i k_2 \theta}w). \label{Reqrel}
\end{equation}

The admissibility conditions of Definition \ref{admis} are precisely the conditions that guarantee that the quotient by this action is smooth in which case it yields a new quasiregular contact manifold with base $\calz_1 \times \calz_2$, contact form $\eta_1 + \eta_2$, contact distribution
given by $\cald_1 \oplus \cald_2$, and Reeb vector field $R_{\eta_1} + R_{\eta_2}.$  This contact structure is exactly the one coming from the 
principal circle bundle obtained by requiring that its curvature form is the pullback of the sum of the two symplectic forms on each base space.
 We obtain new Reeb orbits as equivalence classes of pairs of Reeb orbits one from each of $M_1, M_2.$  When $k_1, k_2$ are different from $1,$ we have a similar contact manifold, except the curvature
is given by pulling back $k_1$ and $k_2$ multiples of the symplectic forms, namely $d\alpha = \pi^* (k_2 \omega_1 + k_1\omega_2)$.  In this case Reeb orbits in the new total space will correspond to pairs, one 
wrapping $k_1$ the other wrapping $k_2$ times (in addition to the multiplicity of the orbit as a Reeb orbit in one of the three spheres).   
In the following, $M_1, M_2$ are both standard three spheres. Index calculations on three dimensional spheres are 
standard, however, we present the following lemma for completeness and also to illustrate the inherent role of the orbifold structure.

Let us consider the contact structure on the quotient of the product of two standard weighted three spheres with weights $p_1, p_2, p_3, p_4.$  As before we take $k_1=\gcd(p_1, p_2),$ $k_2 = \gcd(p_3, p_4),$ and $\overline{p_i}=\frac{p_i}{k_1}$ for $i=1, 2$, and $\overline{p_j}=\frac{p_j}{k_2}$ for $j=3, 4.$ We view this as a 
product of hypersurfaces in $\bbc^4$ with coordinates $(z_1, z_2, z_3, z_4 )$ with $z_j=x_j +iy_j,$ subject to the action (\ref{Reqrel}).  This manifold
is the total space of an orbibundle over an orbifold $S^2\times S^2$ with orbifold singularities at the products of the north and south poles, and for the products
of the north and south poles with copies of $S^2$  These singularities correspond to 
setting one or two of the $z_j$ to $0.$  The Reeb vector field is given by 
\begin{equation}\label{Reebs3}
 p_1 y_1 \partial_{x_1} - x_1 p_1 \partial_{y_1} +  p_2 y_2 \partial_{x_2} - p_2 x_2 \partial_{y_4} + p_3 y_3 \partial_{x_3} - x_3 p_3 \partial_{y_3} +  p_4 y_4 \partial_{x_4}
 - p_4 x_4 \partial_{y_4}  
\end{equation}

and the contact distribution is given by the span of the vectors 
\begin{equation}
  -\frac{1}{p_1}x_2 \partial_{x_1} + \frac{1}{p_1}y_2 \partial_{y_1} +\frac{1}{p_2} x_1 \partial_{x_2} -\frac{1}{p_2}y_1 \partial_{y_2} ~\label{s3condist1} 
\end{equation}

\begin{equation}
  -\frac{1}{p_1} y_2 \partial_{x_1} -\frac{1}{p_1} x_2 \partial_{y_1} +\frac{1}{p_2} y_1 \partial_{x_2} +\frac{1}{p_2} x_1 \partial_{y_2}~\label{s3condist2} 
\end{equation}

\begin{equation}
  -\frac{1}{p_3}x_4 \partial_{x_3} +\frac{1}{p_3} y_4 \partial_{y_3} +\frac{1}{p_4}x_3 \partial_{x_4} -\frac{1}{p_4}y_3 \partial_{y_4} ~\label{s3condist3} 
\end{equation}

and
\begin{equation}
  -\frac{1}{p_3}y_4 \partial_{x_3} -\frac{1}{p_3}x_4 \partial_{y_3} +\frac{1}{p_4}y_3 \partial_{x_4} +\frac{1}{p_4} x_3 \partial_{y_4}.~\label{s3condist4} 
\end{equation}
In the following we restrict ourselves to the case where $p_3=p_4 =k_2.$
To get our hands on an orbit in the quotient, we must, for each time around the fiber, pick an appropriate circle out of the fiber of the torus bundle. 
It is easy to see that the equivalence relation gives us a circle obtained by wrapping around the first circle $k_2$ times and around the second circle $k_1$ times.   
Let us now parametrize the fiber.
We may choose coordinate for a Reeb orbit by:
$$\g(t)=(0,cos(k_2 \overline{p_2}t)+ isin(k_2\overline{p_2}t) ,0, cos(p_2t)+ i sin(p_2 t) ).$$
Now when $t=\frac{1}{\overline{p_2}}$ we have wrapped around the first orbit $k_2$-times and the second one $k_1$ times. 
Here the action is $\frac{1}{p_2}.$  This is the smallest action since we have assumed that $p_2 > p_1.$
What about when the first orbit wraps around more than once in $S^3$? Let us see how to look at such an orbit.
This corresponds to taking
\begin{equation}\label{morbit}
\g(t)=(cos(k_2\overline{p_1}t)+ isin(k_2\overline{p_1}t) ,0,0, cos(\frac{p_1}{m}t)+ i sin(\frac{p_1}{m} t) )
\end{equation}
where $m$ is the multiplicity.  Now when $t=\frac{m}{p_1},$ we wrap around the first orbit $mk_2$ times and the second one $k_1$ times.  As long
as $m< \min \{k_2, p_1 \}$ we do not enter a higer dimensional orbit space.  Similar considerations remain true for $z_1=0.$   
Let us compute the Robbin-Salamon index of the orbits (\ref{morbit}).
To do this we must choose a disk $D$ with boundary $\g$.  
Such a disk can be written explicitly.  We begin by producing a disk in $S^3 \times S^3.$ 
\begin{equation} \label{disk}
(cos(\theta), sin(\theta)e^{2 \pi i k_2\overline{p_2} t}, cos(\theta), sin(\theta)e^{2 \pi i \frac{p_2 t}{m}})
\end{equation}
The above disk clearly has boundary $\g$ (the boundary occurs when $\theta= \frac{\pi}{2}$) and we have $\theta \in [0, \frac{\pi}{2}].$
 To pull back the contact distribution we plug the coordinates into (\ref{s3condist1})-(\ref{s3condist4})
\begin{equation}
  -\frac{1}{p_1}sin(\theta)cos(2 \pi i k_2 \overline{p_2} t ) \partial_{x_1} +\frac{1}{p_1} sin(\theta)sin(2 \pi i k_2\overline{p_2} t ) \partial_{y_1} +\frac{1}{p_2}cos(\theta) \partial_{x_2}  ~\label{trivs3condist1} 
\end{equation}

\begin{equation}
  -\frac{1}{p_1}sin(\theta)sin(2 \pi i k_2\overline{p_2} t) \partial_{x_1} -\frac{1}{p_1}sin(\theta)cos(2 \pi i k_2\overline{p_2} t ) \partial_{y_1}  +\frac{1}{p_2} cos(\theta) \partial_{y_2}~\label{trivs3condist2} 
\end{equation}

\begin{equation}
-\frac{1}{p_3}sin(\theta)cos(2 \pi i \frac{\overline{p_2}t}{m}) \partial_{x_3} +\frac{1}{p_3} sin(\theta)sin(2 \pi i \frac{\overline{p_2} t}{m}) \partial_{y_3} +\frac{1}{p_4}cos(\theta) \partial_{x_4}   ~\label{trivs3condist3} 
\end{equation}

and
\begin{equation}
  -\frac{1}{p_3}sin(\theta)sin(2 \pi i \frac{\overline{p_2} t }{m}) \partial_{x_3} -\frac{1}{p_3}sin(\theta)cos(2 \pi i \frac{\overline{p_2} t }{m}) \partial_{y_3}  + \frac{1}{p_4}cos(\theta) \partial_{y_4}.~\label{trivs3condist4} 
\end{equation}

When $\theta=\frac{\pi}{2},$ these $4$ vectors become 
$$\frac{1}{p_1}(-cos(2 \pi i k_2\overline{p_2} t) \partial_{x_1} + sin(2 \pi i k_2\overline{p_2} t) \partial_{y_1}),$$
$$\frac{1}{p_1} (-sin(2 \pi i k_2 \overline{p_2} t) \partial_{x_1} -cos(2 \pi i k_2\overline{p_2} t) \partial_{y_1}), $$ 
$$\frac{1}{p_3} (-cos(2 \pi i \frac{p_2 t}{m})\partial_{x_3} +sin(2 \pi i \frac{p_2 t}{m})\partial_{y_3}), $$
and $$\frac{1}{p_3}(-sin(2 \pi i \frac{p_2 t}{m})\partial_{x_3} -cos(2 \pi i \frac{p_2 t}{m})\partial_{y_3}) .$$

Disks for the other orbits mapping into branch divisors have a similar expression.  The key point is that we only see vectors corresponding to the 
coordinates which have been set to zero.  

Now we can easily compute the Robbin-Salamon index of these orbits.  Recall that given a path of symplectic matrices, $\Phi(t)$, a number $t$ is called a
crossing if $\Phi(t)$ has an eigenvalue equal to $1.$  To compute the Robbin-Salamon index of a path of symplectic matrices on $[0,T]$ one computes
$$\frac{1}{2}\mbox{signature}(\Gamma(0)) + \sum_{\mbox{crossings }t\neq 0 \mbox{,}T }\mbox{signature}(\Gamma(t)) + \frac{1}{2}\mbox{signature}(\Gamma(T)). $$
Here the crossing form is $$\o(\dot{\Phi}(t)v, v))$$ 
restricted to the subspace on which $\Phi$ has eigenvalues equal to $1.$ 
In this case at each crossing the crossing form is just $\o(v, J_0v)$, so this gives signature $2$ on each two dimensional subspace consisting of 
eigenvectors with eigenvalue $1.$  At crossings the vectors above spanning $\cald$ above become
$-\frac{1}{p_1}\d x_1 ,$
$-\frac{1}{p_1}\d y_1 ,$
$-\frac{1}{l}\d x_3 ,$ 
and
$-\frac{1}{l}\d y_3 .$

Recall that the linearized Reeb flow is of the form 
$$ \left[ \begin{array}{l c c r}
e^{2 \pi i  p_1 t}& 0& 0 &0\\
0& e^{2 \pi i  p_2 t}& 0 &0\\
0&0&e^{2 \pi i k_2 t}& 0\\
0&0&0&e^{2 \pi i k_2 t}
\end{array} \right] 
$$
Each complex block of the matrix looks like 
$$ \left[ \begin{array}{l r}
cos(2 \pi p_j t)& -sin(2 \pi p_j t)\\
sin(2 \pi  p_j t)& cos(2 \pi   p_j t)\\
\end{array} \right]$$ 
The time derivative of each block looks like
$$ \left[ \begin{array}{l r}
-2 \pi  p_j sin(2 \pi  p_j t)& - 2 \pi  p_j cos(2 \pi  p_j t)\\
2 \pi  p_j cos(2 \pi  p_j t)& -2 \pi  p_j sin(2 \pi   p_j t)\\
\end{array} \right]$$ 
At crossings these blocks become
$$ \left[ \begin{array}{l r}
0& - 2 \pi  p_j \\
2 \pi  p_j & 0\\
\end{array} \right].$$ 
 
The crossings which have the first two vectors as $1$-eigenvectors occur at integers multiples of $\frac{1}{k_2p_1},$  and those for the second two occur at 
integer multiples of $\frac{k_1m}{k_2}$.
As we saw above the flow splits into two parts.  That corresponding to the first two coordinates and that corresponding to the second two.  This means the second
part, for multiplicity $m$ is $2mk_1$.  Now we add the normal part.  For orbits of multiplicity $m$ we get contribution $1+ 2\lfloor \frac{m}{p_1}\rfloor.$ 
Therefore for multiplicity $m$ these orbits have Robbin-Salamon index 
$$2k_1m +  2\lfloor \frac{mp_2}{p_1}\rfloor - 1, $$
similarly, setting $p_2$ to zero we obtain 
$$2k_1m +  2\lfloor \frac{mp_1}{p_2}\rfloor - 1.$$ 
We label these orbits $\g_{m,i}$.

The goal here is to distinguish contact structures.  We know that $\frac{p_2-1}{p_2} > \frac{p_1 -1}{p_1}.$ We have then that the index of $\gamma = 2(p_1 -1)$.  This tells us that we have $p_1-1 + p_2 -1$ orbits of index less than $2(p_2 -1)$ 

Theorem \ref{conhomorbibun} gives us a complete picture of the contact homology of the manifolds given by admissible $4$-tuples up to knowing the Robbin-Salamon indices.  Let us spell this out in the case $(p_1, p_2, k_2, k_2)$  In this case there are essentially two different kinds of orbit spaces.  We have $2$-dimensional orbit spaces
which project to $2$-spheres in the base, and we have copies of the whole manifold.  The $2$-dimensional orbits spaces consist of orbits having action
$\frac{k_2m}{p_i} $ for $p_i \not | m.$  The $4$-dimensional orbit spaces consist of orbits of integer action.  
For each $2$-dimensional orbit space, $S_{\frac{k_2m}{p_i}}$ we obtain exactly two orbits contributing to contact homology with grading difference two.  
We denote these orbits $\hat{\gamma_{m,i}}$, and $\check{\gamma_{m_i}}$ corresponding to the maximum and minimum of the Morse function on $S_{\frac{k_2m}{p_i}}$.
For such orbits with action less than $1$ we have grading 
\begin{equation}\label{brgen}
\begin{split}
|\hat{\g_{m,i}}| =\mu_{RS}(\gamma_{m,i})+1 \\
|\check{\g_{m,i}}|=\mu_{RS}(\gamma_{m,i})-1.
\end{split}
\end{equation}
For each $4$-dimensional orbit space we have $4$ generators for contact homology, again corresponding to critical points. We label these 
$$\hat{\g_{m}}, \check{\g_m}, \g^{s_1}_m, \g^{s_2}_m$$ for the maximum, minimum, and two saddle points, respectively.  With a choice of disk $D$ projecting
to the spherical homology class $\Sigma \in H_2(\calz, \bbq)$ we have 
\begin{equation}\label{densegen}
\begin{split}
|\hat{\g_{m}}| =\mu_{RS}(\gamma_{m}, D)+2 \\
|\check{\g_{m}}|=\mu_{RS}(\gamma_{m}, D)-2 \\
| \g^{s_1}_m|= \mu_{RS}(\g_m, D)\\ 
|\g^{s_2}_m|= \mu_{RS}(\g_m, D)
\end{split}
\end{equation}
In (\ref{densegen}) $\mu_{RS}(\gamma_{m}, D)= 2k_2m\langle c_1^{orb}(\calz),(\Sigma)\rangle.$ Moreover for $2$-dimensional orbit spaces with action greater than $1$
by the catenation property of the Robbin-Salamon index, we may decompose the orbit into a part with biggest possible integer action and a part with action
smaller than $1.$  We then add the indices of these two orbits to get the Robbin-Salamon index. Note that for these two dimensional orbit spaces, the tangential part of the flow is a loop, but the normal does not complete a loop, this explains the appearance of the summand $1$ in the above formulae. 
Note that, here, the Robbin-Salamon index is non-decreasing with respect to action.  Thus we may count the number of orbits with index less than $1.$  This will give a count of generators of contact homology of index less than $2(p_1 + p_2 +2)-2.$  From the above discussion there are $p_1 -1 + p_2 -1$ such orbits coming from 
lower dimensional orbit spaces, and then one coming from $\check{g_1}$.  This gives exactly $p_1 + p_2 -1$ orbits in degree less than $2(p_1+ p_2 +1).$ 
The discussion of the previous paragraphs now implies proposition \ref{different}.

\begin{proposition}\label{different}
 Let $(p_1, p_2, l,l)$ and $(p_1', p_2'. l',l')$ be two admissible 4-tuples.  If $p_1+p_2 \neq p_1'+p_2'$, then the corresponding contact manifolds cannot be contactomorphic. 
\end{proposition}

This applies directly to the $Y^{p,q}$ manifolds.  In this case the invariant is $2p-1,$ note that it does not depend on $q.$

As an application let us use the preceeding discussion to distinguish contact structures on the toric contact $5$-manifolds corresponding
to the $4$-tuple $(1, 2k-1, l, l)$ for positive integers $k$, $l$ such that the tuple, $(1, 2k-1, l, l)$, is admissible. Then we see that
$c_1(\mathcal{D})=2k - 2l.$  Let us fix the first Chern class of the contact distribution and see what happens.  We see then that
we must have $$k = \frac{c_1(\mathcal{D}) + 2l}{2}.$$
Now using proposition \ref{different} we see that there are $2k-1$ generators in contact homology of degree less than $4k+2.$

\subsection{A remark on the regular case}
In the regular case the situation is somewhat simpler, but on the other hand there is less information available at first glance.  In this case there is geometrically only one orbit spaces, $\calz$ itself.  To get a handle on the contact homology let us look at the case $(k, k, k-c, k-c)$.  This gives a regular contact manifold with $c_1(\cald)= 2c.$  We choose a basis $L_1, L_2$ of $H_2(\calz, \bbz)$ So that $L_1 = x S_1 + y S_2$, and $L_2$ lifts to a class which evaluates to $0$ under $\pi^* \o$. $x$, $y$, are chosen so that they give action $1$ for a disk that projects to $L_1$.  We define 
$x$, $y$ as follows.  Let $m$ be the smallest number so that $mk \equiv -1 \mod{3}$.  Then we define $x=\frac{km+1}{c}$, $y=x-m$.  It is easy to see that $x$, $y$ satisfy the above properties.  With these choices the grading for contact homology, for orbits of action $N$ is given by
$$|\hat{\g}|= N(2x + 2y) +2$$
$$|\check{\g}|= N(2x + 2y) -2$$
and
$$|\g^{s_j}_N|= N(2x + 2y)$$

In this picture for $N=1$, $\check{\g}$ gives the smallest possible grading.  By varying $k$, we obtain infinitely many distinct contact structures whose contact distribution has the same Chern class.

\subsection{Another way to distinguish contact structures}
As described in Section 2.9.2 of \cite{ElGiHo00} there is another situation where symplectic field theory can be used to distinguish toric contact structures. The following theorem is a generalization to smooth orbifolds of Proposition 2.9.4 of \cite{ElGiHo00}:
 
\begin{thm}\label{EGHgen}
Suppose we have two simply connected quasi-regular toric contact manifolds of Reeb type in dimension $5,$ such that 
each orbifold stratum is non-singular in the sense that its underlying space is a smooth submanifold. Suppose that under the quotient of the Reeb action one of the base manifolds has an exceptional sphere while the other does not, then these two manifolds are not contactomorphic.
 \end{thm}

\begin{proof}
 We show that there is an odd element in the contact homology algebra of one manifold specialized at a class which is not in the other for 
any specialization.  We assume here that all of the weights of the torus action are greater than $1$ for the manifold containing no exceptional spheres. 
As in ~\cite{ElGiHo00} the potential specialized to the Poincar\'{e} dual of an exceptional divisor will give the potential for a standard 
$S^3,$  but then for a chain which lifts to the volume form for this $3$-form there is always a holomorphic curve to kill it as a 
generator for homology specialized at this $3$ class.  Hence this homology contains no odd elements.  
Let us consider first the case where the base is a manifold.  We look at the manifold containing no exceptional sphere.  We must compute the Gromov-Witten potential.  Unfortunately it does not vanish, but, for \emph{any} 
$2$-classes the potential always vanishes. This is because the Gromov-Witten invariant 
$$GW_{A,k}^0(\alpha, \ldots, \alpha) \neq 0$$
for a $2$-dimensional class $\alpha$ only if 
$$ 2k = 4 + 2 c_1(A) + 2k -6 \Leftrightarrow c_1(A)=1$$
But the weights make this impossible.  Thus all coefficients for such curves vanish, and the potential vanishes on $\calz$, hence on $M$.
So for a $3$ class in the contact manifold obtained from integration over the
fiber of a two class, there is no holomorphic curve to kill it.  Hence specialized at such a $3$ class we have an odd generator which does 
not exist in the presence of exceptional spheres.
The orbifold case is similar.  The computation for the Gromov-Witten potential on the manifold with the exceptional
sphere follows from the divisor axiom.  To see that the coefficients for the Gromov-Witten potential vanish in the 
case where there is no exceptional sphere we note the Gromov-Witten invariant is non-zero only if the first Chern class
evaluated on $A$ is equal to one minus the degree shifting number of $\textbf{x},$ which in the absence of exceptional spheres in the stratum in question is impossible.    
\end{proof}
\begin{remark}
 We should mention here that since the base is four dimensional the results of ~\cite{HoLiSi} tell us that we can indeed use the dimension formula
above for computation of the Gromov-Witten invariants for the manifold case.  To adjust for orbifold structure we use extra point conditions since 
all strata in this case are actually smooth manifolds additionally endowed with orbifold structure. 
\end{remark}

\section{The Contact Equivalence Problem}

It is the purpose of this section to prove the equivalence of certain toric contact structures that are inequivalent as toric contact structures. If two toric contact structures $\cald_\bfp$ and $\cald_{\bfp'}$ are equivalent as contact structures, but inequivalent as toric contact structures, it means that there is a contactomorphism $\varphi:M_\bfp\ra{1.6} M_{\bfp'}$ such that $\varphi_*\cald_\bfp=\cald_{\bfp'}$, but there is no $T^3$-equivariant contactomorphism. Then their 3-tori correspond to distinct conjugacy classes of maximal tori in the contactomorphism group \cite{Ler03b,Boy10a}.

A complete answer to the equivalence problem appears to be quite difficult so we restrict ourselves to certain special cases of contact structures that are Seifert $S^1$-bundles (orbibundles) over Hirzebruch surfaces which generally has a non-trivial orbifold structure. In this case we show that certain $T^3$ equivariantly inequivalent contact structures are actually $T^2$ equivariantly equivalent for some subgroup $T^2\subset T^3$.

\subsection{Orbifold Hirzebruch Surfaces}
In this section we study a special class of toric contact structures on $S^3$ bundles over $S^2$ that can be realized as circle orbibundles over orbifold Hirzebruch surfaces. Since the reduction method gives all examples of such toric contact structures, it is important to make contact (no pun intended) with examples that are known in the literature. Here we shall always assume that the quadruple $\bfp$ is admissible.

When working with Hirzebruch surfaces, we often follow \cite{GrHa78} (but with slightly different notation) and represent $S_n$ as the projectivized bundle $S_n=\bbp(\calo(n)+\calo)\ra{1.6} \bbc\bbp^1$ with fibers $L=\bbc\bbp^1$, and sections $E$ with self-intersection number $n$, and $F$ with self-intersection $-n$. These define divisors in $S_n$ and determine a basis for the Picard group ${\rm Pic}(S_n)\approx  H^2(S_n,\bbz)\approx \bbz^2$ which satisfy $E\cdot E=n, E\cdot L=1$ and $L\cdot L=0$. However, when working with symplectic forms it is convenient to use a basis which appear for all admissible complex structures. Thus, it is convenient to treat the even and odd Hirzebruch surfaces separately. The even Hirzebruch surfaces $S_{2n}$ are diffeomorphic to $S^2\times S^2$, so we define $E_0= E-nL$. Then we have
$$E_0\cdot E_0=(E-nL)\cdot (E-nL)= E\cdot E-2nE\cdot L+L\cdot L=2n-2n=0.$$
In this case the Poincar\'e duals $\gra_L$ and $\gra_{E_0}$ are the standard area forms for the two copies of $S^2$. Similarly, the odd Hirzebruch surfaces $S_{2n+1}$ are diffeomorphic to $\bbc\bbp^2$ blown-up at a point which we denote by $\widetilde{\bbc\bbp^2}$. In this case we define $E_{-1}=E-(n+1)L$ which gives $E_{-1}\cdot E_{-1}=-1$. So $E_{-1}$ is an exceptional divisor. Again the Poincar\'e duals $\gra_L$ and $\gra_{E_{-1}}$ represent the standard area forms on a fiber and exceptional divisor, respectively.

As mentioned previously the orbifolds that we encounter are of the form $(X,\grD)$ where $X$ is a smooth algebraic variety and $\grD$ is a branch divisor. Specifically we are interested in the orbifolds $(S_n,\grD)$ where $\grD=\sum_i(1-\frac{1}{m_j})D_j$ where $D_i$ are Weil divisors on $S_n$. We refer to the pair $(S_n,\grD)$ as an {\it orbifold Hirzebruch surface}. We now wish to compute the orbifold canonical divisor in this situation. Since we are working with $\bbq$-divisors, we can express the result in terms of $E_0$ even though it is an honest divisor only on even Hirzebruch surfaces.

\begin{lemma}\label{orbcan}
Let $(S_n,\grD)$ be an orbifold Hirzebruch surface such that $E$ and $F$ are branch divisors both with ramification index $m$. Then the orbifold canonical divisor of $(S_{n},\grD_{m})$ is 
$$K^{orb}_{(S_{n},\grD_{m})}=-\frac{2}{m}E-\frac{2m-n}{m}L =-\frac{2}{m}E_0-2L.$$
Hence, $(S_{n},\grD_{m})$ is a log del Pezzo surface (Fano) if and only if $2m>n$. 
\end{lemma}

\begin{proof}
We know (\cite{GrHa78}, pg 519) that the canonical divisor $K_{S_n}$ of $S_n$ is given $K_{S_n}=-2E+(n-2)L=-2E_0-2L$, and the the orbifold canonical divisor $K^{orb}_{S_n,\grD_{m}}$ satisfies (cf. \cite{BG05}, pg 127)
$$K^{orb}_{S_n,\grD_{m}}= K_{S_n} +(1-\frac{1}{m})\bigl(E+F)\bigr).$$
Now the divisor $E$ has self-intersection $n$, and the divisor $F$ has self-intersection $-n$, and since they both have intersection $1$ with the fiber $L$, we have $(E+F)=2E_0$. Putting this together gives the formula.

The orbifold $(S_{n},\grD_{m})$ is log del Pezzo if and only if the orbifold anticanonical divisor $-K^{orb}_{(S_{n},\grD_{m})}$ is ample, and this happens if and only if $2m>n$ by Nakai's criterion since $E$ and $L$ are effective. 
\end{proof}

\subsection{Toric Contact Structures on $S^2\times S^3$} 
The toric contact structures we describe here are not the most general, but are obtained by setting $\bfp=(j,2k-j,l,l)$. That is, we consider contact structures of the form $\cald_{j,2k-j,l,l}$ where the pair $(k,l)$ is fixed with $k\geq l$, and $j=1,\cdots,k$. Now since $\bfp\in \cala$ we also have $\gcd(j,l)=\gcd(2k-j,l)=1$. We denote the set of $j=1,\cdots,k$ such that $\bfp=(j,2k-j,l,l)$ is admissible by $\calj_\cala=\calj_\cala(k,l)$. The first Chern class of this contact structure is $c_1(\cald_{j,2k-j,l,l})=2(k-l)\grg$ where $\grg$ is a generator of $H^2(M_\bfp,\bbz)\approx \bbz$. So in this case $M_\bfp$ is $S^2\times S^3$. The infinitesimal generator of the circle action is $L_\bfp=jH_1+(2k-j)H_2 -lH_3-lH_4$. Note that this case includes the $Y^{p,q}$ as a special case, namely, $p=k=l$ and $q=k-j$. So $Y^{p,q}$ is $\cald_{p-q,p+q,p,p}$ with $p>q$ and $\gcd(p,q)=1$.

We want to find a suitable Reeb vector field in the Sasaki cone, so we try  $R_{j,k,l}=(2k-j)H_1+jH_2+lH_3+lH_4$ which clearly satisfies the positivity condition $\eta_0(R)>0$. The $T^2$ action generated by $L_\bfp$ and $R_{j,k,l}$ is
$$\bfz\mapsto (e^{i((2k-j)\phi+j\theta)}z_1,e^{i(j\phi+(2k-j)\theta}z_2,e^{il(\phi-\theta)}z_3, e^{il(\phi-\theta)}z_4).$$
Making the substitutions $\psi=\phi-\theta$ and $\chi=j\psi+2k\theta$ gives the action
\begin{equation}\label{genact1}
\bfz\mapsto (e^{i(2(k-j)\psi+\chi)}z_1,e^{i\chi}z_2,e^{il\psi}z_3, e^{il\psi}z_4).
\end{equation}
We define $g_j=\gcd(l,2(k-j))$ and write $2(k-j)=n_jg_j$ and $l=m_jg_j$. Then $\gcd(m_j,n_j)=1$.

\begin{theorem}\label{Hirzorbi}
Consider the contact manifold $(S^2\times S^3,\cald_{j,2k-j,l,l})$ where $1\leq j\leq k$ satisfies $\gcd(j,l)=\gcd(2k-j,l)=1$. Then we have
\begin{enumerate}
\item The quotient space by the circle action generated by the Reeb vector field $R=(2k-j)H_1+jH_2+lH_3+lH_4$ is the K\"ahler orbifold $(S_{n_j},\grD;\gro_{k,l,j})$ where $S_{n_j}$ is a Hirzebruch surface, $\grD$ is the branch divisor
\begin{equation}\label{brdiv}
\grD=(1-\frac{1}{m_j})(E+F),
\end{equation}
and $\gro_{k,l,j}$ is an orbifold symplectic form satisfying $\pi^*\gro_{k,l,j}=d\eta_{k,l,j}$ where $\eta_{k,l,j}$ is the contact 1-form representing $\cald_{j,2k-j,l,l}$ whose Reeb vector field is $R$. 
\item The orbifold structure is trivial ($\grD=\emptyset$) if and only if $l$ divides $2(k-j)$.
\item The induced Sasakian structure is positive if and only if $l>k-j$.
\end{enumerate}
 
\end{theorem}

\begin{proof} For (1) the idea, in the spirit of GIT quotient equals symplectic quotient \cite{Kir84,Ness84}, is to identify the symplectic quotient $\mu^{-1}(0)/T^2$ with a Hirzebruch surface as an analytic subspace of $\bbc\bbp^1\times \bbc\bbp^2$.

After shifting by a constant vector $\bfa=(a_1,a_2)$ the moment map of the $T^2$ action (\ref{genact1}) is
\begin{equation}
\mu(\bfz)=(2(k-j)|z_1|^2+l|z_3|^2+l|z_4|^2-a_1,|z_1|^2+|z_2|^2-a_2).
\end{equation}
We need to choose the constant vector $\bfa$ so that $0$ is a regular value of $\mu$ for all integers $j,l$ such that $0<j\leq k$, and $0<l\leq k$.
Alternatively, it suffices to show that the $T^2$ action on $\mu^{-1}(0)$ defined by Equation (\ref{genact1}) is locally free. This will be true if we choose $a_1,a_2>0$ and $a_1>2(k-j)a_2$.
Following \cite{Aud94} it is convenient to work with the corresponding $\bbc^*\times \bbc^*$ action on $\bbc^2\setminus\{0\}\times \bbc^2\setminus\{0\}$ given by  
\begin{equation}\label{C*act}
\bfz\mapsto (\grt^{n_j}\grz z_1,\zeta z_2,\grt^{m_j}z_3, \grt^{m_j}z_4)
\end{equation}
where $\grt,\zeta\in \bbc^*$. From this we see that the action is free if $z_1z_2\neq 0$ and locally free with isotropy group $\bbz_{m_j}$ on the two divisors obtained by setting $z_1=0$ and $z_2=0$, respectively. It is not difficult to see \cite{Laf81} that $(z_1=0)=E$ and $(z_2=0)=F$. 

We have a commutative diagram 
\begin{equation}\label{orbibranchcover3}
\xymatrix{\bbc^2\setminus\{0\}\times \bbc^2\setminus\{0\} \ar[rdd]^{\pi''} &\ar@{_{(}->}@<.2ex>[l]\ \mu^{-1}(0)\ar[dd]^{\pi'}\ar[rd]^\grr & \\
&&M_{k,l,j}\, .\ar[ld]^\pi\\
&\calz_{k,l,j}&\\
}
\end{equation}
and we want to identify the quotient space $\calz_{k,l,j}$. We know from the general theory (cf. \cite{BG05}) that $\calz_{k,l,j}$ is a projective algebraic orbifold with an orbifold K\"ahler structure. Viewing $\calz_{k,l,j}$ as the $\bbc^*\times \bbc^*$ quotient by the map $\pi''$ of diagram (\ref{orbibranchcover3}), we can identify $\calz_{k,l,j}$ with a subvariety of $\bbc\bbp^1\times \bbc\bbp^2$ as in \cite{Hir51,Laf81} as follows. If we define homogeneous coordinates in $\bbc\bbp^1\times \bbc\bbp^2$ by setting $(w_1,w_2)=(z_3,z_4)$ and $(y_1,y_2,y_3)=(z_2^{m_j}z_3^{n_j},z_2^{m_j}z_4^{n_j},z_1^{m_j})$, we see that $\calz_{k,l,j}$ is represented by the equation
\begin{equation}\label{calzeq}
w_1^{n_j}y_2=w_2^{n_j}y_1.
\end{equation}
As an algebraic variety this identifies $\calz_{k,l,j}$ with the hypersurface in $\bbc\bbp^1\times \bbc\bbp^2$ defined by Equation (\ref{calzeq}) which is the original definition of the Hirzebruch surface $S_{n_j}$. However, the two divisors in $S_{n_j}$ defined by $E=(y_3=z_1^{m_j}=0)$ and $F=(z_2=0 (y_1=y_2=0))$ are both $m_j$-fold branch covers with isotropy group $\bbz_{m_j}$. Thus, we have an orbifold structure on $S_{n_j}$ given by Equation (\ref{brdiv}) which is trivial if and only if $m_j=1$ which happens if and only if $l$ divides $2(k-j)$. 

Furthermore, it follows from the orbifold Boothby-Wang theorem \cite{BG00a} that $\calz_{k,l,j}$ has an orbifold K\"ahler form $\gro_{k,l,j}$ that satisfies $\pi^*\gro_{k,l,j}=d\eta_{k,l,j}$. This proves (1).

For (2) we note that if the orbifold structure is trivial if and only if $m_j=1$ which happens if and only if $g_j=l$ divides $2(k-j)$. To prove (3) we see that Lemma \ref{orbcan} states that $(S_{n_j},\grD_{m_j})$ is log del Pezzo if and only if 
$2l=2m_jg_j>n_jg_j=2(k-j)$. Furthermore, a quasi-regular Sasakian structure is positive if and only if its orbifold quotient is log Fano \cite{BG05}.
\end{proof}

\begin{corollary}\label{posindef}
If $l\leq k-j$,then the toric contact structure $\cald_{j,2k-j,l,l}$ on $S^2\times S^3$ has both a positive and indefinite Sasakian structure in its Sasaki cone.
\end{corollary}

\begin{proof}
By Corollary \ref{s3s2cor} any toric contact structure on $S^2\times S^3$ has a positive Sasakian structure, and by (3) of Theorem \ref{Hirzorbi} the inequality assures us that $\cald_{j,2k-j,l,l}$ has an indefinite Sasakian structure.
\end{proof}

Notice that on subsets of $\calj_\cala(k,l)$ where $g_j$ is independent of $j$, the ramification index $m_j$ is also independent of $j$, so the underlying orbifolds are the same. Thus, it is convenient to view $g_j$ as a map $g:\calj_\cala(k,l)\ra{1.6} \{1,\cdots,l\}$, and we are interested in the level sets of this map. So we decompose $\calj_\cala(k,l)$ into the level sets of $g$ and then further decompose the level sets according to whether $n_j$ is odd or even, that is we define
\begin{flalign*}
g^{-1}(i)_{even}=&\{j\in \calj_\cala(k,l) ~|~g_j=i,~n_j~\text{is even} \} \\
g^{-1}(i)_{odd}=&\{j\in \calj_\cala(k,l) ~|~g_j=i,~n_j~\text{is odd} \}.
\end{flalign*}
We can then decompose the admissible set as a disjoint union
\begin{equation}\label{caljdec}
\calj_\cala(k,l)=\bigsqcup_{i=1}^lg^{-1}(i)_{even}\sqcup g^{-1}(i)_{odd}
\end{equation}

We wish to compute the symplectic form on the base orbifold. 
\begin{lemma}\label{evenHirgro}
Let $j\in g^{-1}(i)_{even}\subset \calj_\cala(k,l)$ for a fixed $i\in\{1,\cdots,l\}$ with $g^{-1}(i)_{even}\neq \emptyset$. Then the symplectic form $\gro_{k,l,i}$ on the quotient $(S_{n_j},\grD_i)$ is independent of $j$ and satisfies
$$[\gro_{k,l,i}]=i\gra_{E_0}+k\gra_L.$$
\end{lemma}

\begin{proof}
We know that $c_1^{orb}(S_{n_j}) =\frac{2}{m_i}\gra_{E_0} +2\gra_L$ and this must pull back to $2(k-l)\grg$. So $\pi^*\gra_{E_0}=m_ik\grg$ and $\pi^*\gra_L=-l\grg$. Now the class $[\gro_{k,l,i}]$ must transcend to $0$ on $S^2\times S^3$. So writing $[\gro]=a\gra_{E_0}+b \gra_L$ and using $im_i=l$, we see that 
$$0=a\pi^*\gra_{E}+b\pi^*\gra_L=\bigl(am_ik-bl\bigr)\grg=m_i\bigl(ak-bi\bigr)\grg.$$
So taking $a=i$ and $b=k$ gives the result.
\end{proof}

\begin{lemma}\label{oddHirgro}
Let $j\in g^{-1}(i)_{odd}\subset \calj_\cala(k,l)$ for a fixed $i\in\{1,\cdots,l\}$ with $g^{-1}(i)_{odd}\neq \emptyset$. Then $i$ is even and the symplectic form $\gro_{k,l,i}$ on the quotient $(S_{n_j},\grD_i)$ is independent of $j$ and satisfies
$$[\gro_{k,l,i}]=i\gra_{E_{-1}}+(k+\frac{i}{2})\gra_L.$$
\end{lemma}

\begin{proof}
First $i$ must be even since $i=g_j=\gcd(l,2(k-j))$ and $n_j=\frac{2(k-j)}{i}$ is odd.
The remainder of the proof is the same as that of Lemma \ref{evenHirgro}, except for odd Hirzebruch surfaces we express the symplectic class in term of the exceptional divisor $E_{-1}=E_0-\frac{1}{2}L$. 
\end{proof}

Whenever possible  we would like to determine the cardinalities $\#g^{-1}(i)_{even}$ and $\#g^{-1}(i)_{odd}$. First as seen above $g^{-1}(i)_{odd}$ is empty when $i$ is odd. Moreover, if $g^{-1}(l)$ is not empty, then $\grD_l=\emptyset$, so the orbifold structure is trivial. The following lemma is due to Karshon \cite{Kar03}

\begin{lemma}\label{cardlem}
The following hold:
\begin{enumerate}
 \item $\#g^{-1}(l)_{even}=\lceil \frac{k}{l}\rceil$,
\item $\#g^{-1}(l)_{odd}=\lceil\frac{2k-l}{2l}\rceil$.
\end{enumerate}
\end{lemma}

\begin{example}\label{Ypqex}
One obtains the $Y^{p,q}$ of \cite{GMSW04a} as a special case of Theorem \ref{Hirzorbi} by putting $k=l=p$ and defining $q=k-j$. The contact structures are then $\cald_{p-q,p+q,p,p}$, and the admissibility conditions boil down to $\gcd(q,p)=1$. Clearly, we have $c_1(\cald_{p-q,p+q,p,p})=0$. When $p$ is odd we have $g_j=1,m_j=p$ and $S_{n_j}=S_{2q}$, whereas if $p$ is even we have $g_j=2,m_j=\frac{p}{2}$ and $S_{n_j}=S_q$ with $q$ odd. So here we have only two non-empty level sets of the map $g$, namely,
\begin{equation}\label{Ypqlevset}
\calj_\cala(p,p)=\begin{cases} g^{-1}(1)_{even}, &\text{if $p$ is odd}; \\
                g^{-1}(2)_{odd} &\text{if $p$ is even}. 
                 \end{cases}
.
\end{equation}
Since the only admissibility condition is $\gcd(q,p)=1$ and $p>q$, we see that the cardinality $\#\calj_\cala(p,p)=\phi(p)$ where $\phi$ is the well known Euler phi function.
For the orbifold canonical divisor Lemma \ref{orbcan} gives
\begin{flalign}\label{orbcanpq}
K^{orb}_{(S_{2q},\grD)}=&-\frac{2}{p}E-\frac{2(p-q)}{p}L=-\frac{2}{p}E_0-2L, \\
K^{orb}_{(S_{q},\grD)}=&-\frac{4}{p}E-\frac{2(p-q)}{p}L=-\frac{4}{p}E_{-1}-2\frac{p+1}{p}L,
\end{flalign}
so these are all log del Pezzo surfaces. The cohomology class of the corresponding symplectic forms are 
$$[\gro_{p,p,1}]=\gra_{E_0}+p\gra_L$$
on the even orbifold Hirzebruch surface $(S_{2q},\grD)$ with ramification index $m_1=p$ when $p$ is odd. For even $p$ we have 
$$[\gro_{p,p,2}]=2\gra_{E_{-1}}+(p+1)\gra_L$$
on the odd orbifold Hirzebruch surface $(S_q,\grD)$ with ramification index $m_2=\frac{p}{2}$. Note that in both cases there  are precisely $\phi(p)$ values taken on by $q$. Note also that $p=2$ implies $q=1$ only, and $m_2=1$, so we have a trivial orbifold structure on $(S_1,\emptyset)=\widetilde{\bbc\bbp^2}$. A relation between the $Y^{p,q}$ toric contact structures and Hirzebruch surfaces was noted by Abreu \cite{Abr09}.
\end{example}

Except for the $Y^{p,q}$ case of Example \ref{Ypqex}, we do not have a general formula for the cardinalities $\#g^{-1}(i)$ for $i\neq l$. Specific cases, of course, are easy to work out.

\begin{example}\label{98}
Consider the case $(k,l)=(9,8)$. We compute $\calj_\cala(9,8)$. The possible values of $j$ are all odd $j\leq 9$, and these all satisfy $\gcd(8,18-j)=1$. Next, we determine $g_j=\gcd(8,2(9-j))$ and $n_j=\frac{2(9-j)}{g}$. So we have $g^{-1}(8)_{even}=\{j=1,9\}$ with a trivial orbifold $(m_8=1)$ on the Hirzebruch surfaces $S_2,S_0$, respectively. We also have  $g^{-1}(8)_{odd}=\{j=5\}$ with a trivial orbifold on $S_1$, and $g^{-1}(4)_{odd}=\{j=3,7\}$ with $m_4=2$ on the odd Hirzebruch surfaces $S_3,S_1$, respectively. Notice that the cardinalities of $g^{-1}(8)_{even}$ and $g^{-1}(8)_{odd}$ agree with Lemma \ref{cardlem}. In total we have $\#\calj_\cala(9,8)=5$.
\end{example}

\subsection{Toric Contact Structures on $X_\infty$}
For the case of $X_\infty$ we consider $\bfp=(j,2k-j+1,l,l)$ with $0<j\leq k$. Here we have $c_1(\cald_{j,2k-j+1,l,l})=\bigl(2(k-l)+1\bigr)\grg$. We consider the Reeb vector field $R_{j,k,l}^\infty=(2k-j+1)H_1+jH_2+lH_3+lH_4$ which is clearly positive. The $T^2$ action generated by this vector field and $L_\bfp=jH_1+(2k-j+1)H_2-lH_3-lH_4$ is
$$\bfz\mapsto (e^{i((2k-j+1)\phi+j\theta)}z_1,e^{i(j\phi+(2k-j+1)\theta}z_2,e^{il(\phi-\theta)}z_3, e^{il(\phi-\theta)}z_4).$$
Making the substitutions $\psi=\phi-\theta$ and $\chi=j\psi+(2k+1)\theta$ gives the action
\begin{equation}\label{regoddact1}
\bfz\mapsto (e^{i(2k-2j+1)\psi+\chi)}z_1,e^{i\chi}z_2,e^{il\psi}z_3, e^{il\psi}z_4).
\end{equation}

Similar to the previous section we define $g_j=\gcd(l,2k-2j+1)$ and write $2k-2j+1=n_jg_j$ and $l=m_jg_j$. Then $\gcd(m_j,n_j)=1$.

\begin{theorem}\label{Hirzorbi2}
Consider the contact manifold $(X_\infty,\cald_{j,2k-j+1,l,l})$ where $1\leq j\leq k$ satisfies $\gcd(j,l)=\gcd(2k-j+1,l)=1$. Then we have 
\begin{enumerate} 
\item The quotient space by the circle action generated by the Reeb vector field $R=(2k-j+1)H_1+jH_2+lH_3+lH_4$ is the K\"ahler orbifold $(S_{n_j},\grD;\gro_{k,l,j})$ where $S_{n_j}$ is an odd Hirzebruch surface, $\grD$ is the branch divisor
\begin{equation}\label{brdiv2}
\grD=(1-\frac{1}{m_j})(z_1=0) +(1-\frac{1}{m_j})(z_2=0),
\end{equation}
and $\gro_{k,l,j}$ is an orbifold symplectic form satisfying $\pi^*\gro_{k,l,j}=d\eta_{k,l,j}$ where $\eta_{k,l,j}$ is the contact 1-form representing $\cald_{j,2k-j+1,l,l}$ whose Reeb vector field is $R$. Here the integers $l,g_j,n_j,m_j$ are all odd.
\item The orbifold structure is trivial ($\grD=\emptyset$) if and only if $l$ divides $2k-j+1$.
\item The induced Sasakian structure is positive if and only if $2l>2(k-j)+1$.
\end{enumerate}
 \end{theorem}

\begin{proof}
The proof is essentially the same as that of Theorem \ref{Hirzorbi}. The details are left to the reader.
\end{proof}

Similar to Corollary \ref{posindef} we have 
\begin{corollary}\label{posindef2} 
If $2l\leq 2(k-j)+1$,then the toric contact structure $\cald_{j,2k-j+1,l,l}$ on $X_\infty$ has both a positive and indefinite Sasakian structure in its Sasaki cone.
\end{corollary}

We denote the set of $j=1,\cdots,k$ such that $\bfp=(j,2k-j+1,l,l)$ is admissible by $\calj_\cala^\infty=\calj_\cala^\infty(k,l)$.
Since the integers $g_j,n_j$ are both odd for all $j\in\calj_\cala(k,l)$, the map $g$ maps the set $\calj_\cala(k,l)$ to the set of positive odd integers less than or equal to $l$. Thus, we have
\begin{equation}\label{glevsetodd}
\calj_\cala^\infty(k,l)=\bigsqcup_{odd~i=1}^lg^{-1}(i)
\end{equation}
Similar to Lemmas \ref{evenHirgro},\ref{oddHirgro} and \ref{cardlem} we find

\begin{lemma}\label{oddHirgro2}
Let $j\in g^{-1}(i)\subset\calj_\cala^\infty(k,l)$. Then  the symplectic form $\gro_{k,l,i}$ on the quotient $S_{n_j}$ is independent of $j$ and satisfies
$$[\gro_{k,l,i}]=i\gra_{E_{-1}}+(k+\frac{i+1}{2})\gra_L.$$
Furthermore, $\#g^{-1}(l)=\lceil\frac{2k-l+1}{2l}\rceil.$
\end{lemma} 

\begin{example}\label{Xinftyex}
We consider the analogue on $X_\infty$ of Example \ref{Ypqex}, so $(k,l)=(p,p)$ with $p$ odd and $j=p-q$. The contact structure is $\cald_{p-q,p+q+1,p,p}$ with $c_1(\cald_{p-q,p+q+1,p,p})=1$. The admissibility conditions are $\gcd(q,p)=1=\gcd(q+1,p)$. The function $g$ satisfies $g_j=\gcd(2q+1,p)$. If $p$ is prime then the set of admissible $q$ is $\{1,\cdots,p-2\}$, and $g=1$ except when $q=\frac{p-1}{2}$ in which case $g=p$. The latter is smooth and corresponds to the trivial orbifold $(S_0,\emptyset)$. This has symplectic class 
$$[\gro_{p,p}]=p\gra_{E_{-1}}+ \frac{3p+1}{2}\gra_L;$$
whereas, the $p-3$ elements in $g^{-1}(1)$ have symplectic class 
$$[\gro_{p,p}]=\gra_{E_{-1}}+ (p+1)\gra_L.$$

For a case when $p$ is not prime consider $p=9$. Then $\calj_\cala^\infty(9,9)=\{2,5,8\}$ which corresponds to $q=7,4,1$. We see that $g^{-1}(9)=\{5\}$ giving the smooth first Hirzebruch surface $(S_1,\emptyset)$ with symplectic class $[\gro_{9,9,9}]=9\gra_{E_{-1}}+14\gra_L$, while, $g^{-1}(3)=\{2,8\}$ giving the orbifold Hirzebruch surfaces $(S_5,\grD)$ with $m_2=3$, and $(S_1,\grD)$ with $m_8=3$, respectively, with symplectic class $[\gro_{9,9,3}]=3\gra_{E_{-1}}+11\gra_L$.
\end{example}

\subsection{Equivalent Contact Structures}
Here we show that certain inequivalent toric contact structures are equivalent as contact structures. The proof uses the fact that the identity map $(S_n,\grD_m)\ra{1.6} (S_n,\emptyset)$ is a Galois cover, and combines this with the work of Karshon \cite{Kar99,Kar03} for the smooth case. 

\begin{theorem}\label{equivcont}
Consider the toric contact structures $(S^2\times S^3,\cald_{j,2k-j,l,l})$ and $(X_\infty,\cald_{j,2k-j+1,l,l})$ of Theorem \ref{Hirzorbi} and \ref{Hirzorbi2}, respectively. Then,
\begin{enumerate}
\item for each fixed $1\leq i\leq l$, the contact structures $\cald_{j,2k-j,l,l}$ are $T^2$-equivariantly isomorphic for all $j\in g^{-1}(i)_{even}$, and the contactomorphism group $\gC\go\gn(\cald_{j,2k-j,l,l})$ has at least $\#g^{-1}(i)_{even}$ conjugacy classes of maximal tori of dimension $3$;
\item for each fixed $1\leq i\leq l$, the contact structures $\cald_{j,2k-j,l,l}$ are $T^2$-equivariantly isomorphic for all $j\in g^{-1}(i)_{odd}$, and the contactomorphism group $\gC\go\gn(\cald_{j,2k-j,l,l})$ has at least $\#g^{-1}(i)_{odd}$ conjugacy classes of maximal tori of dimension $3$;
\item for each fixed $1\leq i\leq l$, the contact structures $\cald_{j,2k-j+1,l,l}$ are $T^2$-equivariantly isomorphic for all $j\in g^{-1}(i)\subset\calj_\cala^\infty(k,l)$, and the contactomorphism group $\gC\go\gn(\cald_{j,2k-j+1,l,l})$ has at least $\#g^{-1}(i)$ conjugacy classes of maximal tori of dimension $3$;
\item The $T^2$-equivariantly isomorphic contact structures in items $(1)-(3)$ are not $T^3$-equivariantly isomorphic.
\end{enumerate}
\end{theorem}

\begin{proof}
The proof of (1)-(3) are quite analogous, so we give the details for (1) only. By Theorem \ref{Hirzorbi} the contact structure is the orbifold Boothby-Wang construction over the symplectic orbifold $(S_{n_j},\gro_{k,l,j})$ and by Lemma \ref{evenHirgro} the form $\gro_{k,l,j}$ only depends on $g_j=i$. Then for each $j\in g^{-1}(i)_{even}$ we consider the Galois cover $\BOne_{n_j}:(S_{n_j},\grD_{m_i})\ra{1.7} (S_{n_j},\emptyset)$ with $n_j$ even, and both spaces having the same K\"ahler form, namely $\gro_{k,l,i}$ of Lemma \ref{evenHirgro}. Now Karshon \cite{Kar03} shows that $(S_{n_j},\emptyset)$ and $(S_{n_{j'}},\emptyset)$ are $S^1$-equivariantly symplectomorphic with the same symplectic form $\gro_{k,l,i}$ (but not the same K\"ahler structure) as long as $0\leq j'=j-2r$ for some non-negative integer $r$. We denote such a symplectomorphism by $K$. Now we have a commutative diagram
\begin{equation}\label{Galoisdiag}
\begin{matrix}(S_{n_j},\grD_{m_i}) &\fract{K_i}{\ra{2.8}} &(S_{n_{j'}},\grD_{m_i}) \\
                         \decdnar{\BOne_{n_j}} &&           \decupar{\BOne^{-1}_{n_{j'}}} \\
                         (S_{n_j},\emptyset)&\fract{K}{\ra{2.8}} &(S_{n_{j'}},\emptyset)
\end{matrix}
\end{equation}
which defines the upper horizontal arrow $K_i$ and shows that it too is an $S^1$-equivariant symplectomorphism. We claim that $K_i$ is also an orbifold diffeomorphism. This follows from Lemma \ref{Kardiv} below. But then as shown in \cite{Ler03b,Boy10a} this symplectomorphism lifts to a $T^2$-equivariant contactomorphism. 

Here, and hereafter, by $g^{-1}(i)$ we mean any of the three sets $g^{-1}(i)_{even},\break g^{-1}(i)_{odd}$, or $g^{-1}(i)\subset \calj_\cala^\infty(k,l)$. 
Since our contact structures are independent of $j\in g^{-1}(i)$ up to isomorphism, we now denote the contact structures of items (1),(2) and (3) by $\cald_{k,l,i,e},\cald_{k,l,i,o}$, and $\cald_{k,l,i,\infty}$, respectively.
To prove (4) we first notice that as in \cite{Kar03} the orbifold symplectomorphism $K_i$ is only $S^1$-equivariant, not $T^2$-equivariant. So the corresponding $2$-tori belong to different conjugacy classes in the group $\gH\ga\gm((\calb,\grD_{m_i}),\gro_{k,l,i})$ of Hamiltonian symplectomorphisms, where $\calb$ is the symplectic orbifold $((S^2\times S^2,\grD_{m_i})\gro_{k,l,i})$ or $((X_\infty,\grD_{m_i})\gro_{k,l,i})$ as the case may be. But then
by Theorem 6.4 of \cite{Boy10a} these lift to non-conjugate $3$-tori in $\gC\go\gn(\cald_{k,l,i,e})$. Hence, the contact structures $\cald_{j,2k-j,l,l}$ for different $j\in g^{-1}(i)$ are inequivalent as toric contact structures. The same holds for $\cald_{j,2k-j+1,l,l}$.
\end{proof}

\begin{lemma}\label{Kardiv}
The Karshon symplectomorphism $K$ of diagram (\ref{Galoisdiag}) leaves the divisors $(z_1=0)$ and $(z_2=0)$ invariant.
\end{lemma}

\begin{proof}
The $T^2$ action on any orbifold Hirzebruch surface $(S_n,\grD_m)$ can be taken as 
\begin{equation}\label{torict2act}
([w_1,w_2]\times [y_1,y_2,y_3])\mapsto ([\grt w_1,w_2]\times [\grt^ny_1,y_2,\grr y_3]) 
\end{equation}
where, as in the proof of Theorem \ref{Hirzorbi}, the coordinates are $(w_1,w_2)=(z_3,z_4)$ and $(y_1,y_2,y_3)=(z_2^{m}z_3^{n},z_2^{m}z_4^{n},z_1^{m})$. By Proposition 4.1 of \cite{Kar99} $K$ is an $S^1$-equivariant symplectomorphism, where the $S^1$ is that generated by $\grr$, and the fixed point set of this action is the disjoint union $(z_1=0)\sqcup (z_2=0)$. But by Proposition 4.3 of \cite{Kar99} $K$ also intertwines the two $S^1$ moment maps. But in both cases these are represented by $\mu_S(\bfz)=|z_1|^2$. So the divisors $(z_1=0)$ and $(z_2=0)$ are left invariant separately by $K$. 
\end{proof}

\begin{remark}\label{traprem}
Using the Delzant theorem for symplectic orbifolds in \cite{LeTo97}, it is straightforward to construct the labeled polytope corresponding to the symplectic orbifold $((S_n,\grD_m),\gro_{k,l,i})$. It is the labeled Hirzebruch trapezoid shown in Figure \ref{labHirtrap}. The Galois cover $\BOne_{S_n}:(S_n,\grD_m)\ra{1.7} (S_n,\emptyset)$ induces a map on this Hirzebruch trapezoid that simply removes the labels on the vertical edges. This implies that the corresponding Karshon graphs \cite{Kar99} are the same. Hence, Theorem 4.1 of \cite{Kar99} easily generalizes to the types of orbifolds consider here, and the symplectomorphism $K_i$ in diagram \ref{Galoisdiag} can be constructed directly from this. 
\end{remark}

\bigskip
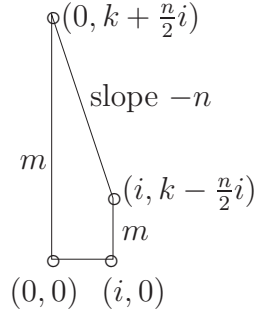
\begin{figure} 
\setlength{\unitlength}{0.00053333in}
\begingroup\makeatletter\ifx\SetFigFont\undefined%
\gdef\SetFigFont#1#2#3#4#5{%
  \reset@font\fontsize{#1}{#2pt}%
  \fontfamily{#3}\fontseries{#4}\fontshape{#5}%
  \selectfont}%
\fi\endgroup%

\ 
\begin{picture}(4156,3039)(-100,-10)
\put(100,2900){$(0,k+\frac{n}{2}i)$}
\path(12,3012)(612,1212)(612,612)(12,612)(12,3012)
\put(387,2112){slope $-n$}
\put(-400,200){$(0,0)$}
\put(500,200){$(i,0)$}
\put(700,1200){$(i,k-\frac{n}{2}i)$}
\put(-50,2900){$\circ$}
\put(-50,520){$\circ$}
\put(520,520){$\circ$}
\put(550,1120){$\circ$}
\put(-300,1500){$m$}
\put(700,800){$m$}

\end{picture}

\caption{The Hirzebruch trapezoid with label $m$ on the two vertical axes corresponding to an orbifold Hirzebruch surface $(S_n,\grD_m)$.}
\label{labHirtrap}
\end{figure}

\subsection{Inequivalence of Contact Structures}
As discussed previously the inequivalence of contact structures is detected first by the first Chern class $c_1(\cald)$ and then by contact homology. The contact structures $\cald_{j,2k-j,l,l}$ and $\cald_{j,2k-j+1,l,l}$ are clearly inequivalent since they live on different manifolds, so adopting Proposition \ref{different} to our current notation we have 


\begin{theorem}\label{contactinequiv2} 
The contact structures $\cald_{j,2k-j,l,l}$ and $\cald_{j',2k'-j',l',l'}$, and $\cald_{j,2k-j+1,l,l}$ and $\cald_{j',2k'-j'+1,l',l'}$ are inequivalent if $k'\neq k$.
\end{theorem}

\begin{remark}
Unfortunately, combining Theorems \ref{equivcont} and \ref{contactinequiv2} does not answer our equivalence problem completely even in our restrictive cases. For example, it would be nice to know that the $\cald_{j,2k-j,l,l}$ are all contact equivalent as $j$ runs through all admissible values from $1$ to $k$. However, our equivalence statement in Theorem \ref{equivcont} only assures equivalence on the level sets of the map $g$, i.e. if we fix $i\in \{1,\cdots,l\}$, then $\cald_{j,2k-j,l,l}$ are equivalent for all $j\in g^{-1}(i)$. Nevertheless, this is enough to give a complete answer to the equivalence problem for the $Y^{p,q}$ of \cite{GMSW04a} which in our notation is $\cald_{p-q,p+q,p,p}$. See Corollary \ref{Ypqcor} below.

The general case with vanishing first Chern class $c_1(\cald)$ was studied in \cite{CLPP05,MaSp05b} where it was shown that all the toric contact structures admit a compatible Sasaki-Einstein metric. These depend on three parameters $a,b,c$ with values in $\bbz^+$ which in our notation is given by $\cald_{a,b,c,a+b-c}$. Except for the subclass $Y^{p,q}$ these fall outside of the scope of our analysis.
\end{remark}

Another statement of contact inequivalence is obtained as an immediate consequence of Theorem \ref{EGHgen}, namely

\begin{corollary}\label{ineq12}
The contact structures on $S^2\times S^3$ described by items (1) and (2) of Theorem \ref{equivcont} are inequivalent.
\end{corollary}

\subsection{Applications of our Results}
Actually more is true than proven in Theorem \ref{equivcont}. Since the base orbifolds $(S_{n_j},\grD_{m_i})$ have the same symplectic form $\gro_{k,l,i}$ for all $j\in g^{-1}(i)$, this lifts to a contactomorphism $\varphi:\cald_{j,2k-j,l,l}\ra{1.6} \cald_{j',2k-j',l,l}$ such that $\varphi^*\eta_{j',2k-j',l,l}=\eta_{j,2k-j,l,l}$.   Since the reduction process carries a preferred complex structure $J$ along with it, the different indices $j$ represent different transverse complex structures $J$. So by using the contactomorphisms $\varphi$ for each admissible $j$, we have

\begin{corollary}\label{bouqcor}
The contact structures $\cald_{k,l,i,e}$, and $\cald_{k,l,i,o}$ on $S^2\times S^3$ and $\cald_{k,l,i,\infty}$ on $X_\infty$ each admit a Sasaki bouquet $\gB_N$ of toric Sasakian structures with $N=\# g^{-1}(i)$. Furthermore, the intersection $\cap_j\grk(\cald,J_j)$ of all the Sasaki cones is an open subset of the Lie algebra $\gt_2$ of a 2-dimensional torus.
\end{corollary}

\begin{example}
Consider the contact structure $\cald_{12,8,2,o}$ on $S^2\times S^3$. The admissible $j$'s are $j=1,3,5,7$ and $g_j=2$ for all $j$. So (ii) of Theorem \ref{equivcont} gives a $T^2$-equivariant contact equivalence of the $T^3$-equivariantly inequivalent toric contact structures $$\cald_{1,23,8,8}\approx \cald_{3,21,8,8}\approx \cald_{5,19,8,8}\approx \cald_{7,17,8,8}.$$
This implies that the number of conjugacy classes $\gn(\cald_{12,8,2,o},3)$ of $3$-tori in $\gC\go\gn(\cald_{12,8,2,o})$ is at least $4$. Furthermore, by (3) of Theorem \ref{Hirzorbi} the induced Sasakian metrics are positive for $\cald_{5,19,8,8}$ and $\cald_{7,17,8,8}$, whereas they are indefinite for the remaining two.

Another contact structure with the same first Chern class as $\cald_{12,8,2,o}$, namely $c_1=8\grg$, is $\cald_{14,10,2,o}$. This consists of the two $T^3$-equivariantly inequivalent toric contact  structures $\cald_{1,27,10,10}$ and $\cald_{7,21,10,10}$, but only for the latter is the induced Sasakian structure positive. In this case we have  $\gn(\cald_{14,10,2,o},3)\geq 2$. Moreover, it follows from Theorem \ref{contactinequiv2} that the contact structures $\cald_{12,8,2,o}$ and $\cald_{14,10,2,o}$ are inequivalent.
\end{example}

As mentioned previously there is one subclass of contact structures on $S^2\times S^3$ where a complete solution to the equivalence problem can be obtained, and they are all known to admit extremal (actually Sasaki-Einstein) metrics.
\begin{corollary}\label{Ypqcor}
The contact structures $Y^{p,q}$ and $Y^{p',q'}$ on $S^2\times S^3$ are inequivalent if and only if $p'\neq p$. Furthermore, the isotopy class of the contact structure defined by $Y^{p,1}$ admits a $\phi(p)$-bouquet $\gB_{\phi(p)}(Y^{p,q})$ such that each of the $\phi(p)$ Sasaki cones admits a unique Sasaki-Einstein metric. Moreover, these Einstein metrics are non-isometric as Riemannian metrics.
\end{corollary}

\begin{proof}
Applying Theorem \ref{equivcont} to Example \ref{Ypqex} shows that $Y^{p,q}$ is contactomorphic to $Y^{p,1}$ for all admissible $q$. But Abreu and Macarini \cite{AbMa10} show that the underlying contact structures  of $Y^{p,1}$ and $Y^{p',1}$ are inequivalent if $p'\neq p$. (This also follows from Proposition \ref{different}.)
By Corollary \ref{bouqcor} there are precisely $\phi(p)$ Sasaki cones in the bouquet. 
The fact that there is a Sasaki-Einstein metric in the Sasaki cone for each $Y^{p,q}$ was first shown in \cite{GMSW04a} while its uniqueness in the Sasaki cone is proved in \cite{CFO07}. 

To prove the last statement suppose to the contrary that the Sasaki-Einstein metrics $g_q$ and $g_{q'}$ are isometric, that is there is a diffeomorphism $\psi$ of $S^2\times S^3$ such that $\psi^*g_{q'}=g_q$. Then by a theorem of Tanno \cite{Tan70} (see also Lemma 8.1.17 of \cite{BG05}) the transformed Sasakian structure $\cals^\psi$ is either $\cals_q$ itself or its conjugate Sasakian structure $\cals^c_q=(-R_q,-\eta_q,-\Phi_q,g_q)$. In either case $\psi$ is a contactomorphism from $Y^{p,q}$ to $Y^{p,q'}$ satisfying $\psi^{-1}\circ \Phi_{q'}\circ \psi=\pm\Phi_q$. But this implies that the corresponding $3$-tori are conjugate which contradicts (4) of Theorem \ref{equivcont}. 
\end{proof}

\begin{example} 
The analogues of the $Y^{p,q}$s on the non-trivial bundle $X_\infty$ are described in Example \ref{Xinftyex}. For simplicity we consider only the case when $p$ is an odd prime in which case there are $p-3$ admissible values for $q$, namely $1,\cdots,\frac{p-1}{2}-1,\frac{p-1}{2}+1,\cdots,p-2$. These inequivalent toric structures are $T^2$-equivariantly equivalent contact structures by (3) of Theorem \ref{equivcont} and their induced Sasakian structures are all positive by (3) of Theorem \ref{Hirzorbi2}. Moreover, the contactomorphism group of this contact structure has at least $p-3$ maximal tori of dimension 3.
\end{example}

\subsection{Some Remarks Concerning Extremal Sasakian Structures}
As with K\"ahler geometry it is of interest to determine the most preferred Sasakian metrics, and as in K\"ahler geometry it seems reasonable to study the critical points of the (now transverse) Calabi functional \cite{BGS06,BGS07b}. In \cite{Boy10b} the first author described bouquets of extremal Sasakian structures on $S^3$-bundles over $S^2$, and the existence of extremal Sasakian metrics on $X_\infty$ was proven. It is not our intention here to delve much further into the existence of such extremal Sasakian structures, but rather discuss briefly their relation to our current work.

Corollary \ref{bouqcor} gives a partial generalization of Theorems 4.1 and 4.2 of \cite{Boy10b}. In this reference it was shown that when the quotient by the Reeb vector field is a smooth manifold, each Sasaki cone in a bouquet admits an extremal Sasakian metric. This follows from well known work of Calabi. It would be interesting to generalize this to the orbifold case by generalizing the method of \cite{GhKo05} to extremal metrics.

As with toric symplectic structures, all toric contact structures of Reeb type admit a compatible Sasakian metric \cite{BG00b}. Furthermore, in our present situation we have

\begin{corollary}\label{torext}
Every toric contact structure on an $S^3$ bundle over $S^2$ admits extremal Sasakian metrics with positive Ricci curvature.
\end{corollary}

\begin{proof}
By Corollary \ref{s3s2cor} every toric contact structure on an $S^3$ bundle over $S^2$ can be realized as an orbifold fibration over a product of weighted projective spaces $\bbc\bbp(\bp_1,\bp_2)\times \bbc\bbp(\bp_3,\bp_4)$ and have positive Sasakian structures. By a result of Bryant \cite{Bry01} all weighted projective spaces admit Bochner-flat metrics and these are extremal \cite{DaGa06}, and the product of extremal K\"ahler metrics is extremal. So these extremal K\"ahler orbifold metrics lift to extremal Sasakian metrics \cite{BGS06} which since $\bbc\bbp(\bp_1,\bp_2)\times \bbc\bbp(\bp_3,\bp_4)$ is log del Pezzo will have a deformation to a Sasakian metric with positive Ricci curvature by Theorem 7.5.31 of \cite{BG05}. Moreover, it follows from a theorem of Calabi \cite{Cal85} that the toric symmetry is retained by these metrics.
\end{proof}

Corollary \ref{torext} implies that each Sasaki cone in every Sasaki bouquet $\gB_N$ of toric contact structures on an $S^3$ bundle over $S^2$ admits extremal Sasakian metrics of positive Ricci curvature. Since the moment cone of any $S^3$ bundle over $S^2$ has exactly 4 facets, recent results of Legendre \cite{Leg09,Leg10} show that every toric contact structure on an $S^3$ bundle over $S^2$ admits at least one and at most seven distinct rays in the Sasaki cone consisting of Sasakian structures whose metrics have constant scalar curvature. Moreover, she shows that for the Wang-Ziller manifolds $M^{1,1}_{k,l}$ with $k>5l$ there exist two distinct rays in the Sasaki cone whose Sasakian metrics have constant scalar curvature. This corresponds to the case $(\bp_1,\bp_2)=1=(\bp_3,\bp_4)$ of Lemma \ref{t2quot}. 

An interesting question which appears to be unknown at this time is whether any Sasaki cones on these toric contact structures are exhausted by extremal Sasaki metrics. There are only several several known cases where this occurs, namely, the standard CR structure on the spheres $S^{2n+1}$ \cite{BGS06}, the Heisenberg group \cite{Boy09}, and $T^2$-invariant contact structures of Reeb type on $T^2\times S^3$ 
\cite{BoTo11}.

\section*{Appendix: Orbifold Gromov-Witten invariants} 
In this appendix, for the convenience of the reader,  we layout some framework and definitions for Gromov-Witten invariants and the so-called Gromov-Witten potential
for compact symplectic manifolds and orbifolds. In this paper we only consider the genus $0$ 
invariants. The Gromov-Witten invariants that we are interested in occur in the base orbifold $\calz$ of an orbibundle $\pi:M\ra{1.6} \calz$ with $\dim(M)=5.$  Hence we are in the semipositive case and we can define the Gromov-Witten invariants as in ~\cite{McSa04}.  Our version of Gromov-Witten theory for symplectic orbifolds
comes from ~\cite{ChRu02}.  The main difference here is that our marked points, and hence our cohomology classes taken as arguments for the
invariant have constraints determining in which orbifold stratum the curves in question lie.  This is an issue since generally some homology classes may live in several strata.  

Roughly speaking a Gromov-Witten is a count of rigid $J$-holomorphic curves representing a homology class 
$A \in H_{2}(M, \bb{Z})/({\rm torsion})$ in general position with marked points in a symplectic manifold $M$ for which the 
marked points are mapped into the Poincar\'{e} duals of certain cohomology classes.  For example we may ask how many spheres, 
(or lines), intersect 2 generic points in $\bb{CP}^n.$  In this case we have $2$ marked points, a top cohomology class, and 
for $A$ the class of a line, $[L].$ 

To make this precise let $(M, \omega)$ be a compact symplectic manifold, let $J$ be an $\omega$-compatible almost complex structure. 
Consider the moduli space 
$$\oldmathcal{M}_{0, k}^{A}(M, J)$$ 
of genus $0$ stable $J$-holomorphic curves into $M$ representing the class $A$ and assume here that we have regularity of the relevant 
linearized Cauchy-Riemann operator for the class $A$, either via some circumstances or by some sort of abstract perturbation argument.  Note also that when we discuss Gromov-Witten theory for compact symplectic manifolds we will consider only 
\emph{somewhere injective} curves.  We define maps
$$ev_j : \oldmathcal{M}_{0, k}^{A}(M, J) \to M$$
and
$$ev : \oldmathcal{M}_{0, k}^{A}(M, J) \to M^{\times k}$$
by evaluation at the marked points.  

By semipositivity the evaluation map represents a submanifold of $M^{\times k}$ of dimension 
$$2n + \langle 2c_1(M) , A \rangle + 2k + 6.$$

Now we define the Gromov-Witten invariant as a homomorphism
$$GW_{A,k}^M: H^*(M) ^{\otimes k} \otimes H_{*}(\oldmathcal{M}_{0,k}^A(M. J)) \to \bb{Q}$$ 
encoded formally as the integral
$$GW^M_{A,k}(\alpha_1, \ldots , \alpha_k) 
:= \int_{\oldmathcal{M}_{0,k}^A(M. J)} ev_1^*\alpha_1 \cup \cdots \cup ev_k^*\alpha_k \cup \pi^*[\oldmathcal{M}_{0,k}^A(M. J)].$$

This is the definition for manifolds.  This definition can be used without the semipositivity condition as long as there is 
a construction of an appropriate object on which to integrate.  Since we will be working in dimension $4$ this will not be an issue.

To extend this definition to orbifolds, there are issues with the definitions of $J$-holomorphic curves, since the idea of a map 
between orbifolds can be a rather sticky issue.  We content ourselves, here, to know that we have a notion of \emph{good} map, and
we will defer to ~\cite{ChRu04,ChRu02} for the analytic set-up.  With that said, we still must extend the definition above so that it makes
sense in a stratified space.  We should also note that the orbifold cohomology of Chen and Ruan is \emph{not} the same as the orbifold 
cohomology mentioned earlier.  This cohomology is simply a way to organize how various classes interact with the 
stratification of the orbifold.  
As in the manifold case we start with a compact symplectic orbifold, $\calz$ and pick a compatible almost complex structure $J.$  We then 
consider moduli spaces of (genus 0) $J$-holomorphic orbicurves into $M$ representing a homology class $A \in H_2(\calz,\bb{Q}).$  But we now
need to consider a new piece of data which organizes the intersection data so that it is compatible with the stratification.  The extra
data will be defined by a $k$-tuple $\textbf{x}$, of orbifold strata, $(\calz_{1}, \ldots, \calz_{k}).$ The length $k$ of $\bfx$ 
should coincide with the number of marked points.  We will write such a moduli space as
$$\oldmathcal{M}_{0,k}^A(\calz, J, \textbf{x}),$$
and require that the evaluation takes the $j$-th marked point into $\calz_{j}.$
The compactification is similar to the manifold case, and consists of stable maps with the obvious adjustments, the caveat being that
we must choose our lift to an orbicurve.  After an appropriate construction of cycles as in the manifold case, Chen and Ruan use 
a virtual cycle construction, so we can define this invariant as in the smooth case above, but we integrate over (the compactification of)
$\oldmathcal{M}_{0,k}^A(\calz, J, \textbf{x}).$  We will write these invariants
$$GW^{\calz}_{A,k, \textbf{x}}(\alpha_1, \ldots , \alpha_k).$$

Another key difference is that this moduli space differs from the predicted dimension in the 
smooth case by a factor of $-2\iota(\textbf{x}),$ the so-called \emph{degree shifting number}. (Again for the definition see ~\cite{ChRu02}.)
The Gromov-Witten invariants satisfy a list of axioms developed by Kontsevich and Manin \cite{KoMa94,KoMa97}.  We will not list all of the axioms, but will mention only some 
which are used in the text.  We use the orbifold notation, for a manifold we would just delete $\textbf{x}$ from the notation, setting
$\iota(\textbf{x})=0.$
\begin{enumerate}
 \item [i.] \textbf{Effective}: $GW_{A,k, \textbf{x}}^{\calz}(\alpha_1, \ldots, \alpha_n) = 0$ as long as $\omega(A) < 0.$
 \item [ii.] \textbf{Grading}:  $GW_{A,k, \textbf{x}}^{\calz}(\alpha_1, \ldots, \alpha_n) \neq 0$ only if
   $$\sum_j deg(\alpha_j) = dim(\calz) + 2c_1(A) + 2k -6 -2\iota(\textbf{x}).$$
 \item [iii.] \textbf{Divisor}: Let $\textbf{x}^j = \textbf{x}$ with the $j$th component removed.  Suppose that for each 
each component of $\textbf{x}, x_i$ that if $x_i$ is mapped into the orbifold singular locus, that that stratum 
is non-singular as a variety.  If $deg(\alpha_n) =2$ then 
$$GW_{A,k, \textbf{x} }^{\calz} (\alpha_1, \ldots, \alpha_n) = (\int_{A}\alpha_n )GW_{A,k-1, \textbf{x}^{n} }^{\calz}(\alpha_1, \ldots, \alpha_{n-1}).$$
\end{enumerate}
  
Now we are in a position to define the Gromov-Witten potential.  This is a generating function which gives a formal power series 
whose coefficients give Gromov-Witten invariants.  It is a way to organize all the information from these invariants into one big
package.  

We give the definition here for the manifold case.  Pick a basis of $H^2(M),$ $a_1, \ldots, a_n$, for a vector $t$ and
a cohomology class $a,$ write $a:=a_t= \sum_i t_i a_i.$  
\begin{defn}Let $(M, \omega)$ , $J$ be as above.  Define the genus $0$ \textbf{Gromov-Witten Potential} as
 $$\textbf{f}(a_t) = \sum_A \sum_k \frac{1}{k!} GW_{A, k}^{M}(a_t, \ldots, a_t) z^{c_1(A)}.$$
The corresponding formula for orbifolds is obtained by accounting for the vector $\textbf{x}.$
\end{defn}

\newcommand{\etalchar}[1]{$^{#1}$}
\def\cprime{$'$} \def\cprime{$'$} \def\cprime{$'$} \def\cprime{$'$}
  \def\cprime{$'$} \def\cprime{$'$} \def\cprime{$'$} \def\cprime{$'$}
  \def\cdprime{$''$} \def\cprime{$'$} \def\cprime{$'$} \def\cprime{$'$}
  \def\cprime{$'$}
\providecommand{\bysame}{\leavevmode\hbox to3em{\hrulefill}\thinspace}
\providecommand{\MR}{\relax\ifhmode\unskip\space\fi MR }
\providecommand{\MRhref}[2]{%
  \href{http://www.ams.org/mathscinet-getitem?mr=#1}{#2}
}
\providecommand{\href}[2]{#2}

\end{document}